\documentclass{amsart}

\usepackage{graphicx}
\usepackage{hyperref}
\usepackage{cite}
\usepackage[utf8]{inputenc}
\usepackage[english]{babel}
\usepackage{amsmath}
\usepackage{amsfonts}
\usepackage{amssymb}
\usepackage{xcolor}
\usepackage{amsthm}
\usepackage[all]{xy}
\usepackage{enumitem}

\newtheorem{thm}{Theorem}[section]
\newtheorem{lemma}[thm]{Lemma}
\newtheorem{prop}[thm]{Proposition}
\newtheorem{cor}[thm]{Corollary}

\theoremstyle{definition}
\newtheorem{defn}[thm]{Definition}
\newtheorem{example}[thm]{Example}

\theoremstyle{remark}
\newtheorem{rem}[thm]{Remark}

\numberwithin{equation}{section}

\newcommand{\abs}[1]{\lvert#1\rvert}

\DeclareMathOperator{\ann}{Ann}
\DeclareMathOperator{\Add}{Add}
\DeclareMathOperator{\add}{add}
\DeclareMathOperator{\Ker}{Ker}
\DeclareMathOperator{\Coker}{Coker}
\DeclareMathOperator{\Img}{Im}
\DeclareMathOperator{\Mat}{Mat}
\DeclareMathOperator{\Hom}{Hom}
\DeclareMathOperator{\End}{End}
\DeclareMathOperator{\Soc}{Soc}
\DeclareMathOperator{\Rad}{Rad}
\DeclareMathOperator{\I}{I}

\newcommand{\sm}{\sigma[M]}
\newcommand{\annl}{\ann^\ell}
\newcommand{\annr}{\ann^r}
\newcommand{\ess}{\leq^\text{ess}}
\newcommand{\dleq}{\leq^\oplus}

\newcommand{\blankbox}[2]{%
  \parbox{\columnwidth}{\centering
    \setlength{\fboxsep}{0pt}%
    \fbox{\raisebox{0pt}[#2]{\hspace{#1}}}%
  }%
}

\begin{document}

\title{$\mathfrak{m}$-Baer and $\mathfrak{m}$-Rickart lattices}

\author{Mauricio Medina-B\'arcenas}
\address{Facultad de Ciencias, Universidad Nacional Aut\'onoma de M\'exico, Circuito Exterior,
	Ciudad Universitaria 04510, Ciudad de México, M\'exico}
\email{mmedina@ciencias.unam.mx}

\thanks{The first author was supported by the grant ``Programa de Becas Posdoctorales en la UNAM 2021'' from the Universidad Nacional Aut\'onoma de M\'exico (UNAM)}

\author{Hugo Rinc\'on Mej\'ia}
\address{Facultad de Ciencias, Universidad Nacional Aut\'onoma de M\'exico, Circuito Exterior,
	Ciudad Universitaria 04510, Ciudad de México, M\'exico}
\email{hurincon@ciencias.unam.mx}



\date{}



\begin{abstract}
	In this paper we introduce the notions of Rickart and Baer lattices and their duals. We show that part of the theory of Rickart and Baer modules can be understood just using techniques from the theory of lattices. For, we use linear morphisms introduced by T. Albu and M. Iosif. We focus on a submonoid with zero $\mathfrak{m}$ of the monoid of all linear endomorphism of a lattice $\mathcal{L}$ in order to give a more general approach and apply our results in the theory of modules. We also show that $\mathfrak{m}$-Rickart and $\mathfrak{m}$-Baer lattices can be characterized by the annihilators in $\mathfrak{m}$ generated by idempotents as in the case of modules.
\end{abstract}

\maketitle

\section{Introduction}\label{intro}

The study of rings and their modules is closely related with the study of lattices. For any ring $R$, every $R$-module $M$ has associated a partial ordered set with arbitrary suprema and infima, namely the lattice of submodules of $M$. This lattice of submodules $\Lambda(M)$, is a complete upper-continuous modular lattice for any $R$-module $M$. The greatest element is $M$ and $0$ is the lowest. The supremum is given by the sum of submodules and the infimum is given by the intersection. Much of the algebraic structure of $M$ is codified by its lattice of submodules. Classical concepts in modules are concepts which can be defined in any complete modular lattice. For example, essential submodule, small submodule, finitely generated submodule, finitely cogenerated submodule, Noetherian module, Artinian module, distributive module and so on \cite{calugareanu2013lattice}. This is why sometimes the language and the techniques from lattice theory gives a better context for algebraic properties. Also, an advantage of working with lattices is the ``point-free" setting that it provides which allows us to apply the results to more general situations. For example, it is known that the subobjects of an object in a complete and cocomplete abelian category is a complete modular lattice \cite[Ch. IV]{stenstromrings}. The intention of this paper is to carry the notions of Baer and Rickart modules to lattices. We want to show that these concepts and its implications mostly depend just of the ordered structure of the lattice of submodules and its relation with the homomorphisms. For, we will use linear morphisms between lattices which were introduced by  T. Albu and M. Iosif in \cite{albu2013category}.

Given a ring $R$ and $X\subseteq R$, the left annihilator of $X$ is given by the left ideal $\annl(X)=\{r\in R\mid rx=0\;\text{ for all }\;x\in X\}$. A ring $R$ is said to be \emph{Baer} (resp. \emph{left Rickart}) if $\annl(X)$ (resp. $\annl(x)$) is generated by an idempotent as a left ideal for every $X\subseteq R$ (resp. $x\in R$). We have that every idempotent $e^2=e$ in a ring $R$ induces a decomposition on $R$ as left $R$-module $R=Re\oplus R(1-e)$ and every direct summand of $_RR$ is generated by an idempotent. Moreover, $\annl(x)$ for an $x\in R$ can be seen as the kernel of the $R$-endomorphism $\_\cdot x:R\to R$ (multiply by $x$ from the right). Hence,
 
 \begin{prop}
 	The following conditions are equivalent for a ring $R$:
 	\begin{enumerate}[label=\emph{(\alph*)}]
 		\item $R$ is a Baer (resp. left Rickart) ring.
 		\item $\annl(X)$ (resp. $\annl(x)$) is a direct summand of $_RR$ for every $X\subseteq R$ (resp. $x\in R$).
 		\item $\bigcap_{f\in X}\Ker f$ (resp. $\Ker f$) is a direct summand of $_RR$ for every $X\subseteq \End_R(R)$ (resp. $f\in \End_R(R)$).
 	\end{enumerate}
 \end{prop}

It is the last item which allows us to generalize the notion of Baer and Rickart rings to modules, as it was made by G. Lee, S.T. Rizvi and C. Roman in \cite{rizvibaer, leerickart}. An $R$-module $M$ with endomorphism ring $\End_R(M)$ is said to be \emph{Baer} (resp. \emph{Rickart}) if $\bigcap_{f\in X}\Ker f$ (resp. $\Ker f$) is a direct summand of $M$ for every $X\subseteq \End_R(M)$ (resp. $f\in \End_R(M)$). As it might be expected, the dual notions can be defined similarly. The concepts of dual-Rickart and dual-Baer modules were studied also by Lee, Rizvi and Roman \cite{lee2011dual} and by D.K. T\"ut\"unc\"u and R. Tribak in \cite{tutuncu2010dual} respectively. More recently, these notions have been taken to the categorical context by S. Crivei and A. K\"or in \cite{crivei2016rickart}. From these works, it is clear that, independently of the context of generality, the lattice of subojects is playing a principal rol. It is our intention to revel this rol and show how the lattice techniques are enough to recover much of the theory presented in \cite{crivei2016rickart,leerickart,lee2011dual} and \cite{rizvibaer}.

In the theory of Baer and Rickart modules (or Baer and Rickart objects), it is used all the endomorphims of a given module. In our case, given a lattice $\mathcal{L}$ and the monoid $\End_{lin}(\mathcal{L})$ of all linear endomorphism of $\mathcal{L}$, we will use submonoids of $\End_{lin}(\mathcal{L})$ with certain characteristics and we will develop the theory of Rickart and Baer lattices relative to these submonoids. On one hand this gives a more general context for the theory and on the other hand, with this approach we will be able to apply our results to the case of modules and recover the known results. Since we are more interested in module theory, our applications will be made for modules, nevertheless it is not difficult to see that many of them also apply to the context of abelian categories.

The paper is divided as follows: Section \ref{intro} is this introduction. In Section \ref{pre}, we present the linear morphisms and some properties of them which will be needed along the paper. In particular we study the analogous of the projections on direct summands of a module. We also show that an idempotent linear morphism on a lattice $\mathcal{L}$ induces a decomposition of $\mathcal{L}$ as in the case of modules (Proposition \ref{idemcomp}). The Section \ref{mbrl} is the main part of this manuscript. Here we introduce the concepts of $\mathfrak{m}$-Baer and $\mathfrak{m}$-Rickat lattices and their duals. The prefix $\mathfrak{m}$ will stand for a submonoid with zero of the monoid $\End_{lin}(\mathcal{L})$ of all linear endomorphisms of a lattice $\mathcal{L}$. We show that for an $R$-module $M$, the endomorphism ring $\End_R(M)$ can be seen as a submonoid of $\End_{lin}(\Lambda(M))$ where $\Lambda(M)$ is the lattice of submodules of $M$. With this identification we are able to recover many results of Baer and Rickart modules from the results in lattices (see Proposition \ref{lricmric}). We determine when a initial interval of a $\mathfrak{m}$-Baer (or $\mathfrak{m}$-Rickart or their duals) lattice inherits the property (Proposition \ref{complbaer} and Proposition \ref{compldbaer}). It is proved that the (resp. arbitrary) intersection of complements in a $\mathfrak{m}$-Rickart lattice (resp. $\mathfrak{m}$-Baer lattice) is a complement provided that the $\mathfrak{m}$ contains all the projections (Proposition \ref{baerricscip} and Proposition \ref{riccipssp}). Later, we study the behaviour of the linear endomorphisms of these lattices. We give a lattice-counterpart of the condition $D_2$ known for modules \cite{mohamedcontinuous} and we use the condition $C_2$ defined in \cite{albu2016conditions} to characterize $\mathfrak{m}$-Rickart lattices and dual-$\mathfrak{m}$-Rickart lattices (Proposition \ref{ricd2} and Proposition \ref{dricc2}). Since $\mathfrak{m}$ is a monoid with zero, it possible to define annihilators. Then, $\mathcal{L}$ is an $\mathfrak{m}$-Rickart lattice if and only if the right annihilator of each $\varphi\in\mathfrak{m}$ is generated by an idempotent and $\ker\varphi$ is $\mathfrak{m}$-$\mathcal{L}$-generated. (Corollary \ref{ricendoric} and its dual Corollary \ref{dricendodric}). In a similar way we characterize $\mathfrak{m}$-Baer lattices and their dual (Proposition \ref{baercar} and Proposition \ref{dbaercar}). At the end of this section we study when an $\mathfrak{m}$-Rickart lattice is $\mathfrak{m}$-Baer. In Section \ref{ksing}, we translate to lattices the concepts of $\mathcal{K}$-nonsingularity, $\mathcal{T}$-nonsingularity and their duals \cite[Definition 9.4]{crivei2016rickart}. With this, we are able to characterize $\mathfrak{m}$-Baer and dual-$\mathfrak{m}$-Baer lattices. The main result of this section states that $\mathcal{L}$ is $\mathfrak{m}$-Baer and $\mathfrak{m}$-$\mathcal{K}$-cononsingular if and only if $\mathcal{L}$ is $\mathfrak{m}$-$\mathcal{K}$-nonsingular and satisfies the condition $C_1$ (Corollary \ref{baercarK} and its dual Corollary \ref{dbaercarT}). In the last section, we study indecomposable Rickart lattices and their products. It is show that an indecomposable lattice $\mathcal{L}$ is Rickart if and only if $\mathcal{L}=2$, the complete lattice of two elements (Lemma \ref{ricind2}). As aplications we give some conditions to know when a Rickart $R$-module is simple or semisimple (Corollary \ref{ricmodsimp} and Corollary \ref{ricmodss}). Finally, we present some results about the direct product of Rickart lattices.

\section{Preliminaries}\label{pre}

Given a module $M$, the set $\Lambda(M)$ of submodules of $M$ is an upper-continuous complete modular lattice where the supremum is given by the sum of submodules and the infimum by the intersection of submodules \cite[Ch. III]{stenstromrings}. An element $a$ in a complete lattice $(\mathcal{L},\leq,\bigvee,\bigwedge,\mathbf{0},\mathbf{1})$ is \emph{complemented} if there exists $b\in\mathcal{L}$ such that $a\wedge b=\mathbf{0}$ and $a\vee b=\mathbf{1}$ where $\mathbf{0}$ and $\mathbf{1}$ denote the least and the greatest elements of $\mathcal{L}$ respectively. The element $b$ is called \emph{a complement of $a$ in $\mathcal{L}$}. In the lattice $\Lambda(M)$, $\mathbf{0}=0$ and $\mathbf{1}=M$. This implies that a submodule $N\in\Lambda(M)$ is a complement if and only if $N$ is a direct summand of $M$. Hence, an $R$-module $M$ with endomorphism ring $\End_R(M)$ is \emph{Baer} (resp. \emph{Rickart}) if and only if $\bigcap_{f\in X}\Ker f$ (resp. $\Ker f$) has a complement in $\Lambda(M)$ for every $X\subseteq \End_R(M)$ (resp. $f\in \End_R(M)$). 

On the other hand, given a morphism $f:M\to N$ of $R$-modules, $f$ induces a $\bigvee$-semilattice morphism $f_\ast:\Lambda(M)\to\Lambda(N)$ given by $f_\ast(L)=f(L)$ for all $L\in\Lambda(M)$. Note that $f_\ast(0)=0$ and $f_\ast(L+\Ker f)=f_\ast(L)$ for all $L\in\Lambda(M)$. Also, $f_\ast$ induces an isomorphism of lattices between $\Lambda(M/\Ker f)$ and $\Lambda(f(M))$. In \cite{albu2013category}, the authors capture the properties of this morphism in the following definition.

\begin{defn}\label{linmor}
	Let $\varphi:\mathcal{L}\to \mathcal{L}'$ be a map between the bounded lattices $\mathcal{L}, \mathcal{L}'$.  $\varphi$ is called a \emph{linear morphism} if there exist $\ker_\varphi\in\mathcal{L}$, called the \emph{kernel} of $\varphi$ and $a'\in\mathcal{L}'$ such that 
	\begin{enumerate}
		\item $\varphi(x)=\varphi(x\vee \ker_\varphi)$ for all $x\in\mathcal{L}$.
		\item $\varphi$ induces an isomorphism of lattices $\overline{\varphi}:[\ker_\varphi,\mathbf{1}]\to [\mathbf{0},a']$ given by $\overline{\varphi}(x)=\varphi(x)$ for all $x\in[\ker_\varphi,\mathbf{1}]$.
	\end{enumerate}
\end{defn}

\begin{rem}\label{linmorjoins}
    In the case that the lattices are complete, we will assume that the isomorphism in item (2) of Definition \ref{linmor} is an isomorphism of complete lattices. Note that in this case a linear morphism commutes with arbitrary joins \cite[Proposition 1.3]{albu2013category}.
\end{rem}

\begin{example}
    Let $M$ and $N$ be two $R$-modules. If $f:M\to N$ is a morphism of $R$-modules, then $f_\ast:\Lambda(M)\to \Lambda(N)$ is a linear morphism of lattices.
\end{example}

\begin{prop}
	Let $\mathcal{L}$ be a bounded modular lattice. Let $x\in\mathcal{L}$ be an element with complement $x'$. Then, the map $\pi_x:\mathcal{L}\to \mathcal{L}$ given by $\pi_x(a)=(a\vee x')\wedge x$ is a linear morphism.
\end{prop}

\begin{proof}
	Let $a\in\mathcal{L}$. Then, $\pi_x(a\vee x')=((a\vee x')\vee x')\wedge x=(a\vee x')\wedge x=\pi_x(a)$, thus $ x'$ is a kernel for $\pi_x$.
    Note that $\overline{\pi_x}:[x',\mathbf{1}]\to [\mathbf{0},x]$ has the inverse $ (-)\vee x' : [\mathbf{0},x]\to [x',\mathbf{1}]$. Because if $b\in [x',\mathbf{1}] $ then $\overline{\pi_x}(b)=(x'\vee b)\wedge x=b\wedge x $, and so
    $(b\wedge x)\vee x' =b\wedge(x\vee x')=b$, by modularity. If $c\in [\mathbf{0} ,x] $ then $\overline{\pi_x}(c\vee x')=((c\vee x')\vee x')\wedge x=(c\vee x')\wedge x =c\vee(x\wedge x')=c$, by modularity.
    Thus, $\pi_x$ is a linear morphism and $\ker_{\pi_x}=x'$.
\end{proof}

\begin{cor}
A bounded modular lattice $\mathcal{L}$ is boolean if and only if $a\wedge\_:\mathcal{L}\to \mathcal{L}$  is a linear morphism for all $a\in\mathcal{L}. $ 
\end{cor}

\begin{proof}
$\Rightarrow$ Suppose $\mathcal{L}$ is boolean and let $a'$ denote the complement of $a$.  For each $x\in \mathcal{L}$ we have that $x\wedge a=\left(x\vee a'\right)\wedge a$. Thus $a\wedge\_=\pi_a$ which is a linear morphism.

$\Leftarrow$ Let $a\in\mathcal{L}$ and $k=\ker_{a\wedge\_}$. Then, $a\wedge\left(x \vee k \right) = a\wedge x$ for all $x\in \mathcal{L}$, and $a\wedge k =0$. Since $\overline{a\wedge\_}:[k,\mathbf{1}] \to [\mathbf{0},a]$  is an isomorphism and $(k\vee a)\wedge a = a$, then $k\vee a=\mathbf{1}$. Thus $k$ is a complement of $a$. 
Now, if $b\wedge a =0$  with $b\in \mathcal{L}$, then  $b\leq k $ because $k$ is the kernel of $a\wedge \_$.
It follows from \cite[Ch. III Proposition 4.4]{stenstromrings} that $\mathcal{L}$ is boolean.
\end{proof}

\begin{defn}
	Let $\mathcal{L}$ be a bounded modular lattice and $x\in\mathcal{L}$ be an element with a complement. The linear morphism $\pi_x$ is called the \emph{projection on $x$}.
\end{defn}

\begin{rem}
    Given an element $x$ with complement in a lattice $\mathcal{L}$ is possible that $x$ has more than one complement. Therefore, $\pi_x$ depends of the complement of $x$ which is taken.
\end{rem}

\begin{lemma}\label{kerpi}
	Let $\mathcal{L}$ be a bounded modular lattice and $x,y\in\mathcal{L}$  elements with complements $x'$ and $y'$ respectively. Then $\ker_{\pi_y\pi_x}=(x\wedge y')\vee x'$.
\end{lemma}

\begin{proof}
	It follows from the proof of \cite[Lemma 2.1]{albu2013category} that $\ker_{\pi_y\pi_x}=\overline{\pi_x}^{-1}(\ker_{\pi_y}\wedge x)$. Since $\pi_x$ induces the canonical isomorphism between $[x',\mathbf{1}]$ and $[\mathbf{0},x]$, it follows that
	\[\ker_{\pi_y\pi_x}=\overline{\pi_x}^{-1}(\ker_{\pi_y}\wedge x)=\overline{\pi_x}^{-1}(y'\wedge x)=(y'\wedge x)\vee x'.\]
\end{proof}

\begin{lemma}\label{modpix}
	Let $M$ be a module and $M=N\oplus N'$ . Let $\pi:M\to N$ and $\iota:N\to M$ be the canonical projection and the  inclusion respectively. Then $(\iota\pi)_\ast=\pi_N:\Lambda(M)\to \Lambda(M)$.
\end{lemma}

\begin{proof}
	We just have to check the equality $\iota\pi(L)=(L+N')\cap N$ for all $L\leq M$ where $N'$ is a complement of $N$. Let $n=\ell+n'$ with $n\in N$, $n'\in N'$ and $\ell\in L$. Then, $\ell=n-n'$. This implies that $n=\iota\pi(\ell)$. Therefore, $(L+N')\cap N\subseteq \iota\pi(L)$. Now, let $\ell\in L$. Then, there exist $n\in N$ and $n'\in N'$ such that $\ell=n+n'$. Hence, $n=\ell-n'\in(L+N')\cap N$. Since $n=\iota\pi(\ell)$, $\iota\pi(L)\subseteq (L+N')\cap N$.
\end{proof}

Given a lattice $\mathcal{L}$, the set of all linear morphisms $\varphi:\mathcal{L}\to\mathcal{L}$ is a monoid with zero using the composition of functions. This set will be denoted by $\End_{lin}(\mathcal{L})$ and its elements will be called linear endomorphisms of $\mathcal{L}$.

\begin{prop}\label{idemcomp}
	Let $\mathcal{L}$ be a bounded modular lattice and $\varphi\in\End_{lin}(\mathcal{L})$. If $\varphi$ is idempotent, then $\mathbf{0}=\ker_\varphi\wedge\varphi(\mathbf{1})$ and $\mathbf{1}=\ker_\varphi\vee\varphi(\mathbf{1})$.
\end{prop}

\begin{proof}
	Since $\varphi$ is a linear morphism, $\varphi$ induces an isomorphism $\overline{\varphi}:[\ker_\varphi,\mathbf{1}]\to[\mathbf{0},\varphi(\mathbf{1})]$. Write $z=\ker_\varphi\vee\varphi(\mathbf{1})$. Then,
	\[\overline{\varphi}(z)=\varphi(z)=\varphi(\ker_\varphi\vee\varphi(\mathbf{1}))=\varphi\varphi(\mathbf{1})=\varphi(\mathbf{1})=\overline{\varphi}(\mathbf{1}).\]
	Hence, $z=\mathbf{1}$. Now, write $y=\ker_\varphi\wedge\varphi(\mathbf{1})$. Then $y\in[0,\varphi(\mathbf{1})]$. Consider $\overline{\varphi}^{-1}(y)\in[\ker_\varphi,\mathbf{1}]$. It follows that
	\[\mathbf{0}=\varphi(y)=\varphi\overline{\varphi}\overline{\varphi}^{-1}(y)=\varphi\varphi\overline{\varphi}^{-1}(y)=\varphi\overline{\varphi}^{-1}(y)=y.\]
\end{proof}

\begin{prop}\label{exmorf}
	Let $\mathcal{L}$ be a bounded modular lattice and $x,x',y\in\mathcal{L}$ with $x'$ a complement of $x$. Suppose that there exists a linear morphism $\varphi:[\mathbf{0},x]\to[\mathbf{0},y]$. Then $\varphi$ can be extended to a linear morphism $\widehat{\varphi}:\mathcal{L}\to \mathcal{L}$.
\end{prop}

\begin{proof}
	Define $\widehat{\varphi}:\mathcal{L}\to\mathcal{L}$ as $\widehat{\varphi}(a)=\varphi(a\wedge x)$. Note that $\widehat{\varphi}(\ker_\varphi\vee x')=\varphi((\ker_\varphi\vee x')\wedge x)=\varphi(\ker_\varphi)=\mathbf{0}$. The morphism $\varphi$ induces an isomorphism $\overline{\varphi}:[\ker_\varphi,x]\to[\mathbf{0},\varphi(x)]$. On the other hand, there is a canonical isomorphism $[x',\mathbf{1}]\to[\mathbf{0},x]$ given by $x\wedge\_$. Therefore, there are isomorphisms
	\[\xymatrix{[\ker_\varphi\vee x',\mathbf{1}]\ar[r]_\cong ^{x\wedge\_} & [\ker_\varphi,x]\ar[r]_\cong^{\overline{\varphi}} & [\mathbf{0},\varphi(x)].}\]
	Note that $\widehat{\varphi}(\mathbf{1})=\varphi(\mathbf{1}\wedge x)=\varphi(x)$. Thus $\widehat{\varphi}$ is a linear morphism and $\ker_{\widehat{\varphi}}=\ker_\varphi\vee x'$.
\end{proof}

\begin{defn}
	Let $\mathcal{L}$ be a lattice and $x\in\mathcal{L}$. It is said that $x$ is \emph{fully invariant in $\mathcal{L}$} if $\varphi(x)\leq x$ for all $\varphi\in\End_{lin}(\mathcal{L})$.
\end{defn}

\begin{lemma}
	Let $\mathcal{L}$ be a complete lattice and $\{x_i\}_I$ be a family of elements of $\mathcal{L}$. If each $x_i$ is fully invariant in $\mathcal{L}$, then so is $\bigvee_{i\in I}x_i$.
\end{lemma}

\begin{proof}
	Let $\varphi\in\End_{lin}(\mathcal{L})$. It follows from \cite[Proposition 1.3]{albu2013category} and Remark \ref{linmorjoins} that 
	\[\varphi\left(\bigvee_{i\in I}x_i \right)=\bigvee_{i\in I}\varphi(x_i)\leq\bigvee_{i\in I}x_i.\]
\end{proof}

\begin{prop}\label{fipi}
	Let $\mathcal{L}$ be a complete modular lattice and $\{a_i\}_I$ be an independent family such that $\mathbf{1}=\bigvee_{i\in I}a_i$.  
	\begin{enumerate}
		\item If $I$ is finite, then $x\leq\bigvee_{i\in I}\pi_{a_i}(x)$ for all $x\in\mathcal{L}$.
		\item If $\mathcal{L}$ is upper-continuous, then $x\leq\bigvee_{i\in I}\pi_{a_i}(x)$ for all $x\in\mathcal{L}$.
	\end{enumerate}
\end{prop}

\begin{proof}
(1) We will prove it first for $n=2$. Then $\mathbf{1}=a_1\vee a_2$ and $\mathbf{0}=a_1\wedge a_2$. 
Let us compute the following:
\[x\vee\pi_{a_1}(x)=x\vee((x\vee a_2)\wedge a_1)=(x\vee a_2)\wedge(x\vee a_1).\]
Analogously,
\[x\vee\pi_{a_2}(x)=(x\vee a_1)\wedge(x\vee a_2).\]
Note that $x\leq\pi_{a_1}(x)\vee\pi_{a_2}(x)$ if and only if $\pi_{a_1}(x)\vee\pi_{a_2}(x)=\pi_{a_1}(x)\vee\pi_{a_2}(x)\vee x=(x\vee a_1)\wedge(x\vee a_2)$. Suppose that $x\nleq\pi_{a_1}(x)\vee\pi_{a_2}(x)$. Then $\pi_{a_1}(x)\vee\pi_{a_2}(x)\neq(x\vee a_1)\wedge(x\vee a_2)$. Therefore, we have the following pentagon
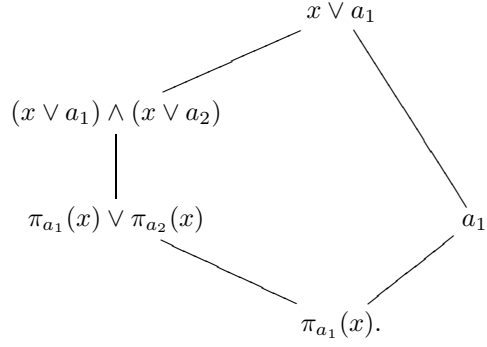
\[\xymatrix{& x\vee a_1\ar@{-}[dl]\ar@{-}[ddr] & \\ (x\vee a_1)\wedge(x\vee a_2)\ar@{-}[d] & & \\ \pi_{a_1}(x)\vee\pi_{a_2}(x)\ar@{-}[dr] & & a_1\ar@{-}[dl] \\ & \pi_{a_1}(x). & }\]
Which is a contradiction. Thus, $x\leq\pi_{a_1}(x)\vee\pi_{a_2}(x)$. Now suppose that the result is true for any complete modular lattice and any independent family of length at most $n-1$. Set $b=\bigvee_{i=2}^na_i$. Then $\mathbf{0}=a_1\wedge b$ and $\mathbf{1}=a_1\vee b$. We have that $\pi_b(x)=(x\vee a_1)\wedge b$. Let $2\leq j\leq n$. Then,
\[\pi_b(x)\vee a_j=((x\vee a_1)\wedge b)\vee a_j=(x\vee a_1\vee a_j)\wedge b.\]
This implies that $\pi_b(x)\vee a_2\vee\cdots a_{n-1}=(x\vee a_1\vee a_2\vee\cdots a_{n-1})\wedge b$. Hence

\begin{equation*}
\begin{split}
(\pi_b(x)\vee a_2\vee\cdots a_{n-1})\wedge a_n & =((x\vee a_1\vee a_2\vee\cdots a_{n-1})\wedge b)\wedge a_n \\
& =(x\vee a_1\vee a_2\vee\cdots a_{n-1})\wedge a_n \\
& =\pi_{a_n}(x)
\end{split}
\end{equation*}

Thus, for any $1\leq j\leq n$, 
\[\left( \pi_b(x)\vee \left(\bigvee_{i\neq j}^{2\leq i\leq n}a_i \right)\right)\wedge a_j =\pi_{a_j}(x). \]
Note that the left hand of the equality is the projection of $\pi_b(x)$ in $a_j$ in the lattice $[\mathbf{0},b]$. By induction hypothesis
\[\pi_b(x)\leq \bigvee_{j=2}^n\left(\left( \pi_b(x)\vee \left(\bigvee_{i\neq j}^{2\leq i\leq n}a_i \right)\right)\wedge a_j \right)=\bigvee_{i=2}^n\pi_{a_i}(x).\]
On the other hand, also by induction hypothesis $x\leq\pi_{a_1}(x)\vee\pi_b(x)$. Thus,
\[x\leq\bigvee_{i=1}^n\pi_{a_i}(x).\]

(2) For $\ell>0$, let $B=\{b_\ell\}$ be the family of elements of $\mathcal{L}$ given by $b_\ell=a_1\vee\cdots\vee a_\ell$. Then $B$ is a directed family and $\bigvee B=\bigvee_{i\in I}a_i=\mathbf{1}$. Since $\mathcal{L}$ is upper-continuous,
\[x=x\wedge\mathbf{1}=x\wedge\bigvee B=\bigvee\{x\wedge b_\ell\mid b_\ell\in B\}.\]
We have that $\mathbf{0}\leq x\wedge b_\ell\leq a_1\vee\cdots\vee a_\ell$. By item (1), 
\[x\wedge b_\ell\leq \bigvee_{i=1}^\ell\pi_{a_i}(x\wedge b_\ell)\leq\bigvee_{i=1}^\ell\pi_{a_i}(x).\]
Thus, $x\leq\bigvee_{i\in I}\pi_{a_i}(x)$.
\end{proof}

The next result appears in \cite[Lemma 1.3]{albu2021strongly} without a proof. We present a proof of it as a corollary of last proposition.

\begin{cor}\label{fidis}
	Let $\mathcal{L}$ be a complete modular lattice and $x\in\mathcal{L}$. If $x$ is fully invariant in $\mathcal{L}$ then $x=\bigvee_{i=1}^n(x\wedge a_i)$ for any independent family $\{a_1,...,a_n\}$ in $\mathcal{L}$ such that $\mathbf{1}=\bigvee_{i=1}^na_i$. Moreover, $\pi_{a_i}(x)=x\wedge a_i$ for all $1\leq i\leq n$. In addition, if $\mathcal{L}$ is upper-continuous then the result is true for any infinite family.
\end{cor}

\begin{proof}
	Let $1\leq i\leq n$. By hypothesis, $\pi_{a_i}(x)\leq x$. Then
	\[\pi_{a_i}(x)\leq x\wedge a_i\leq \left(x\vee\bigvee_{i\neq j}a_j \right)\wedge a_i=\pi_{a_i}(x).\]
	Thus, $\pi_{a_i}(x)=x\wedge a_i$ for all $1\leq i\leq n$. By Proposition \ref{fipi},
	\[x\leq\bigvee_{i=1}^n\pi_{a_i}(x)=\bigvee_{i=1}^nx\wedge a_i.\]
\end{proof}

\section{$\mathfrak{m}$-Rickart and $\mathfrak{m}$-Baer lattices}\label{mbrl}

Given an $R$-module $M$ and an endomorphism $f:M\to M$, there is a linear morphism $f_\ast:\Lambda(M)\to\Lambda(M)$ induced by $f$. Then, there is a homomorphism of monoids with zero $(-)_\ast:\End_R(M)\to\End_{lin}(\Lambda(M))$. Let $\mathfrak{E}_M$ denote the image of $\End_R(M)$ under $(-)_\ast$. Then $\mathfrak{E}_M$ is a submonoid with zero of $\End_{lin}(\Lambda(M))$.

\begin{defn}
	Let $\mathcal{L}$ be a complete lattice and let $\mathfrak{m}$ be a submonoid with zero of $\End_{lin}(\mathcal{L})$.
	\begin{itemize}
		\item $\mathcal{L}$ is called \emph{$\mathfrak{m}$-Baer} if $\bigwedge_{\varphi\in X}\ker_\varphi$ has a complement in $\mathcal{L}$ for all $X\subseteq\mathfrak{m}$.
		\item $\mathcal{L}$ is called \emph{$\mathfrak{m}$-Rickart} if $\ker_\varphi$ has a complement in $\mathcal{L}$ for all $\varphi\in\mathfrak{m}$.
	\end{itemize}
	If the submonoid we are considering is $\End_{lin}(\mathcal{L})$, we will omit the $\mathfrak{m}$.
\end{defn}


\begin{defn}
	Let $\mathcal{L}$ be a complete lattice and let $\mathfrak{m}$ be a submonoid with zero of $\End_{lin}(\mathcal{L})$.
	\begin{itemize}
		\item $\mathcal{L}$ is called \emph{dual-$\mathfrak{m}$-Baer} if $\bigvee_{\varphi\in X}\varphi(\mathbf{1})$ has a complement in $\mathcal{L}$ for all $X\subseteq\mathfrak{m}$.
		\item $\mathcal{L}$ is called \emph{dual-$\mathfrak{m}$-Rickart} if $\varphi(\mathbf{1})$ has a complement in $\mathcal{L}$ for all $\varphi\in\mathfrak{m}$.
	\end{itemize}
	If the submonoid we are considering is $\End_{lin}(\mathcal{L})$, we will omit the $\mathfrak{m}$.
\end{defn}

\begin{rem}
	It is clear that (dual-)$\mathfrak{m}$-Baer implies (dual-)$\mathfrak{m}$-Rickart. Note that every complemented lattice is $\mathfrak{m}$-Baer and dual-$\mathfrak{m}$-Baer for every submonoid $\mathfrak{m}$. For example, the lattice of submodules of a semisimple $R$-module.
\end{rem}

The next result follows from the above definition and it will be the key to apply our results in the context of modules.

\begin{prop}\label{lricmric}
	Let $M$ be an $R$-module and $\Lambda(M)$ the lattice of submodules of $M$. The following conditions are equivalent:
	\begin{enumerate}[label=\emph{(\alph*)}]
	    \item $\Lambda(M)$ is (dual-)$\mathfrak{E}_M$-Rickart (resp. (dual-)$\mathfrak{E}_M$-Baer).
	    \item $M$ is a (dual-)Rickart (resp. (dual-)Baer) module.
	\end{enumerate}
\end{prop}

\begin{proof}
	We will only prove the Rickart case, the others are similar. (a)$\Rightarrow$(b) Let $f:M\to M$ be an endomorphism of $M$. Consider the linear morphism $f_\ast:\Lambda(M)\to \Lambda(M)$. By hypothesis, $\ker_{f_\ast}$ is a complement in $\Lambda(M)$. Note that $\ker_{f_\ast}=\Ker f$, hence $\Ker f$ is a direct summand of $M$.
	
	(b)$\Rightarrow$(a) Let $\varphi\in\mathfrak{E}_M$. Then, there exists an $R$-endomorphism $f:M\to M$ such that $f_\ast=\varphi$. Since $M$ is Rickart, $\Ker f=\ker_{\varphi}$ is direct summand of $M$, that is, $\ker_{\varphi}$ is complemented in $\Lambda(M)$. 
\end{proof}

The following example shows that if we take the monoid $\End_{lin}(\Lambda(M))$ then the implication (b)$\Rightarrow$(a) of Proposition \ref{lricmric} might not be true.

\begin{example}
	Let $K$ be a field. Consider the ring $R=\left(\begin{smallmatrix}
	K & K \\
	0 & K
	\end{smallmatrix}\right)$ and the left $R$-module $M=\left(\begin{smallmatrix}
	0 & K \\
	0 & K
	\end{smallmatrix}\right)$. Since $\End_R(M)\cong K$, $M$ is a Rickart module. On the other hand, the lattice of submodules of $M$ is $\Lambda(M)=\{0, N=\left(\begin{smallmatrix}
	0 & K \\
	0 & 0
	\end{smallmatrix}\right), M\}$. Consider the following linear morphism $\varphi:\Lambda(M)\to \Lambda(M)$ given by $\varphi(0)=0$, $\varphi(N)=0$ and $\varphi(M)=N$, represented in the figure below. Then $\ker_\varphi=N$ which is not a complement in $\Lambda(M)$.
\end{example}

\begin{figure}[h]\label{fig1}
	\[\xymatrix{\bullet\ar@{-}[d]\ar[rrd] & & \bullet\ar@{-}[d] \\ \bullet\ar@{-}[d]\ar[rrd] & & \bullet\ar@{-}[d] \\ \bullet\ar[rr] & & \bullet}\]
	\caption{}
\end{figure}

\begin{rem}
	Note that the lattice of submodules of the $\mathbb{Z}$-module $\mathbb{Z}_4$ is isomorphic to the lattice $\Lambda(M)$ of the previous example. Consider the homomorphism $f:\mathbb{Z}_4\to \mathbb{Z}_4$ given by multiply by 2. Then $f_\ast=\varphi$, the linear morphism in Figure 1.
\end{rem}


\textbf{Notation:} Let $\mathcal{L}$ be a complete modular lattice and $a,x\in\mathcal{L}$. There are two canonical linear morphisms $\iota_x:[\mathbf{0},x]\to \mathcal{L}$ the inclusion, and $\rho_a:\mathcal{L}\to[a,\mathbf{1}]$ given by $\rho_a(y)=a\vee y$. Notice that if $a$ has a complement $a'$, then the interval $[a',\mathbf{1}]$ is canonically isomorphic to $[\mathbf{0},a]$ and so $\pi_a=\iota_{a}(\_\wedge a)\rho_{a'}$.

Given a lattice $\mathcal{L}$ and a submonoid with zero,  $\mathfrak{m}\subseteq\End_{lin}(\mathcal{L})$, we will say that $\mathfrak{m}$ \emph{contains all the projections} if $\pi_a\in\mathfrak{m}$ for every complement $a\in\mathcal{L}$.

\begin{rem}
	Given an $R$-module $M$, the submonoid $\mathfrak{E}_M$ of $\End_{lin}(\Lambda(M))$ contains all the projections by Proposition \ref{modpix}.
\end{rem}

\begin{prop}
Let $\mathcal{L}$ be a bounded modular lattice and $\mathfrak{m}\subseteq\End_{lin}(\mathcal{L})$ be a submonoid containing all the projections. The following conditions are equivalent:
\begin{enumerate}[label=\emph{(\alph*)}]
\item $\mathcal{L}$ is $\mathfrak{m}$-Rickart and for all $a\in\mathcal{L}$ there exists (a unique) $b\in\mathcal{L}$ and an isomorphism $\theta:[a,\mathbf{1}]\to[\mathbf{0},b]$ such that $\iota_b\theta\rho_a\in\mathfrak{m}$. 
\item $\mathcal{L}$ is complemented (boolean).
\end{enumerate}
\end{prop}

\begin{proof}
(a)$\Rightarrow$(b)Let us take $a,b\in\mathcal{L}$ as in the hypothesis. Then the linear morphism,
\[\xymatrix{\mathcal{L}\ar[r]^-{\rho_a} & [a,\mathbf{1}]\ar[r]^\theta_\cong & [\mathbf{0},b]\ar[r]^-{\iota_b} & \mathcal{L}}\]
is in $\mathfrak{m}$ and has kernel $a$. Since $\mathcal{L}$ is $\mathfrak{m}$-Rickart, $a$ has
a complement in $\mathcal{L}$. Now suppose that $b$ is unique an let $a'$ be a complement of $a$. Then, there is a canonical isomorphism $[a,\mathbf{1}]\cong[\mathbf{0},a']$ given by $\_\wedge a'$. Therefore, $\iota_{a'}(\_\wedge a')\rho_a=\pi_{a'}\in\mathfrak{m}$. By hypothesis, $b=a'$. Since $L$ is modular and every element has a unique complement, $\mathcal{L}$ is boolean \cite[Ch. II, Sec. 1, Corollary 3]{gratzer2002general}.

(b)$\Rightarrow$(a) Since $\mathcal{L}$ is complemented, $\mathcal{L}$ is $\mathfrak{m}$-Rickart. If $a\in\mathcal{L}$ and $a'$ is a complement of
$a,$  then $a'\wedge\_:[a,\mathbf{1}]\to[\mathbf{0},a']$ is an isomorphism, and $\iota_{a'}(a'\wedge\_)\rho_{a}=\pi_{a'}\in\mathfrak{m}$. If $\mathcal{L}$ is boolean, then $a'$ is unique.
\end{proof} 

\begin{cor}
    The following conditions are equivalent for a module $M$:
    \begin{enumerate}[label=\emph{(\alph*)}]
        \item $M$ is Rickart and for all $N\leq M$ there exists a (unique) $L\leq M$ such that $M/N\cong L$.
        \item $M=\bigoplus_IS_i$ with $S_i$ simple ($S_i\ncong S_j$ for all $i\neq j\in I$).
    \end{enumerate}
\end{cor}

\begin{prop}
	Let $\mathcal{L}$ be a complete modular lattice and $x\in \mathcal{L}$ be a fully invariant element. If $\mathcal{L}$ is Rickart and every endomorphism $\varphi\in\End_{lin}([\mathbf{0},x])$ can be extended to an endomorphism $\widehat{\varphi}\in\End_{lin}(\mathcal{L})$, then $[\mathbf{0},x]$ is a Rickart lattice.
\end{prop}

\begin{proof}
	Let $\varphi:[\mathbf{0},x]\to[\mathbf{0},x]$ be a linear morphism. By hypothesis, $\varphi$ can be extended to a linear morphism $\widehat{\varphi}:\mathcal{L}\to \mathcal{L}$. Then $\ker_{\widehat{\varphi}}$ is a complement in $\mathcal{L}$. Since the restriction of $\widehat{\varphi}$ to $[\mathbf{0},x]$ is $\varphi$, $\ker_\varphi=\ker_{\widehat{\varphi}}\wedge x$. Let $y$ be a complement of $\ker_{\widehat{\varphi}}$ in $\mathcal{L}$. Then, $\mathbf{1}=\ker_{\widehat{\varphi}}\vee y$ and $\mathbf{0}=\ker_{\widehat{\varphi}}\wedge y$. It follows from Corollary \ref{fidis} that \[x=x\wedge(\ker_{\widehat{\varphi}}\vee y)=(x\wedge\ker_{\widehat{\varphi}})\vee (x\wedge y)=\ker_\varphi\vee(x\wedge y).\]
	It is clear that $\ker_\varphi\wedge(x\wedge y)=\mathbf{0}$. Therefore, $\ker_\varphi$ is a complement in $[\mathbf{0},x]$. Thus $[\mathbf{0},x]$ is a Rickart lattice.
\end{proof}


\begin{lemma}
Let $\mathcal{L}$ be a complete modular lattice. The following conditions are equivalent for $a\in\mathcal{L}$: 
\begin{enumerate}[label=\emph{(\alph*)}]
    \item $a$ is a complement in $\mathcal{L}$.
    \item the inclusion $\iota_a:[\mathbf{0},a]\to \mathcal{L}$ splits.
    \item the projection $\rho_a:\mathcal{L}\to[a,\mathbf{1}]$ splits.
\end{enumerate}
\end{lemma} 
\begin{proof}
(a)$\Rightarrow$(b) Suppose that  $a$ is a complement in $\mathcal{L},$ with complement $a'$. Consider the projection $\pi_{a}$ over $a$ and let us see that it is a splitting for the inclusion of $[\mathbf{0},a]$  in $\mathcal{L}.$ By modularity, we have that
\[\begin{array}{ccccc}
[\mathbf{0},a] & \rightarrowtail & \mathcal{L} & \overset{\pi_{a}}{\rightarrow} & [\mathbf{0},a]\\
x & \mapsto & x & \text{\ensuremath{\mapsto}} & \left(x\vee a'\right)\wedge a=x.
\end{array}\]

(a)$\Rightarrow$(c) Now, for $\rho_a$ consider $\varphi:[a,\mathbf{1}]\to \mathcal{L}$ defined as $\varphi(x)=x\wedge a'$. Then $\varphi(\mathbf{1})=a'$. It is clear that $\varphi$ induces an isomorphism $\overline{\varphi}:[a,\mathbf{1}]\to[\mathbf{0},\varphi(\mathbf{1})]$. Thus $\varphi$ is a linear morphism with $\ker_\varphi=a$. It follows that $\varphi\circ\rho_a=Id_{[a,\mathbf{1}]}$.

(b)$\Rightarrow$(a) Assume that $\varphi:\text{\ensuremath{\mathcal{L\rightarrow\mathit{[\mathbf{0},a]}}}}$ is a splitting for $\iota_a,$  i.e. $\varphi\circ\iota_a=Id_{[\mathbf{0},a]}.$  We will see that $k=\ker_\varphi$  is a complement for $a.$  For, $\mathbf{0}=\varphi(k)=\varphi\left(\left(k\wedge a\right)\vee k\right)=\varphi\left(k\wedge a\right)=k\wedge a.$  On the other hand, as $\bar{\varphi}:[k,\mathbf{1}]\rightarrow [\mathbf{0},a]$ is an isomorphism and $\bar{\varphi}{\left(\mathbf{1}\right)=a=\varphi\left(a\right)=\varphi}\left(a\vee k\right),$ then $a\vee k=\mathbf{1}.$ 

(c)$\Rightarrow$(a) Assume that $\varphi:[a,\mathbf{1}]\to\mathcal{L}$ is a splitting for $\rho_a$  i.e. $\rho_a\circ\varphi=Id_{[a,\mathbf{1}]}.$  We will see that $b=\varphi(\mathbf{1})$  is a complement for $a.$  For, $\mathbf{1}=\rho_a(\varphi(\mathbf{1}))=b\vee a$. On the other hand, as $\bar{\varphi}:[a,\mathbf{1}]\rightarrow [\mathbf{0},b]$ is an isomorphism and $\overline{\varphi}^{-1}(a\wedge b)=\overline{\varphi}^{-1}(a)\wedge \overline{\varphi}^{-1}(b)=a\wedge (a\vee b)=a=\overline{\varphi}^{-1}\left(\mathbf{0}\right)$. Thus, $\mathbf{0}=a\wedge b$.
\end{proof}

\begin{cor}
Let $\mathcal{L}$ be a complete modular lattice and $a,b\in\mathcal{L}$. If $a$ is a complement in $[\mathbf{0},b]$ and $b$ is a complement in $\mathcal{L}$ then $a$ is a complement in $\mathcal{L}. $
\end{cor}

\begin{prop}\label{compintric}
Let $\mathcal{L}$ be a complete modular lattice and $a\in\mathcal{L}$ be an element with complement $a'$. Let $\mathfrak{m}$ be a submonoid of $\End_{lin}(\mathcal{L})$ and let $\mathfrak{n}$ be a submonoid of $\End_{lin}([\mathbf{0},a])$ such that $\iota\psi\pi_a\in\mathfrak{m}$ for every element $\psi\in\mathfrak{n}$. If $\mathcal{L}$ is $\mathfrak{m}$-Rickart, then $[\mathbf{0},a]$ is a $\mathfrak{n}$-Rickart lattice.
\end{prop}

\begin{proof}
Let $\psi:[\mathbf{0},a]\longrightarrow [\mathbf{0},a]$ be a linear morphism in $\mathfrak{n}$. Let us take the projection $\pi_{a}:\mathcal{L\text{\ensuremath{\rightarrow}}L}$ and let $\pi_{a}^{\upharpoonright}:\mathcal{L\mathit{\longrightarrow [\mathbf{0},a]}}$ denote the corestriction to $[\mathbf{0},a]$ of $\pi_a.$ By hypothesis $\begin{array}{ccccccc}
\mathcal{L} & \overset{\pi_{a}^{\upharpoonright}}{\rightarrow} & [\mathbf{0},a] & \overset{\psi}{\rightarrow} & [\mathbf{0},a] & \overset{\iota}{\rightarrow} & \mathcal{L}\end{array}$  is a linear morphism in $\mathfrak{m}$ with kernel $k,$ say, and let us denote $b=\ker_\psi$. We have that $k=\left(\bar{\pi}\right)^{-1}\left(b\right).$ As $\mathcal{L}$ is $\mathfrak{m}$-Rickart, $k$ is a complement in $\mathcal{L}.$  Since $\left(a'\vee b\right)\wedge a=b$ with $a'$ a complement of $a$, $k=a'\text{\ensuremath{\vee}}b.$ Let $k'$ denote a complement of $k$ in $\mathcal{L}.$ Then,
\[\psi\left(a\right)=\left(\psi\pi_{a}\right)\left(1\right)=\left(\psi\pi_{a}\right)\left(k\vee k'\right)=\left(\psi\pi_{a}\right)\left(a'\vee b\vee k'\right)=\psi\left(b\vee k'\right).\]
As $\overline{\psi}:[b,a]\to[\mathbf{0},\psi\left(a\right)]$ is an isomorphism, then $b\vee k'=a.$ Thus, $\left(a\wedge k'\right)\vee b=a\wedge\left(k'\vee b\right)=a.$ Besides, $\left(k'\wedge a\right)\wedge b\leq k'\wedge b\leq k'\wedge\left(b\vee a'\right)=k'\wedge k=0.$ Hence, $k'\wedge a$ is a complement of $b$ in $[\mathbf{0},a].$ Thus $[\mathbf{0},a]$ is $\mathfrak{n}$-Rickart.
\end{proof}

We have the dual version of last proposition.

\begin{prop}
	Let $\mathcal{L}$ be a complete modular lattice and $a\in\mathcal{L}$ be an element with complement $a'$. Let $\mathfrak{m}$ be a submonoid of $\End_{lin}(\mathcal{L})$ and let $\mathfrak{n}$ be a submonoid of $\End_{lin}([\mathbf{0},a])$ such that $\iota\psi\pi_a\in\mathfrak{m}$ every element $\psi\in\mathfrak{n}$. If $\mathcal{L}$ is dual-$\mathfrak{m}$-Rickart, then $[\mathbf{0},a]$ is a dual-$\mathfrak{n}$-Rickart lattice.
\end{prop}

\begin{cor}
	Let $M$ be an $R$-module and $N$ be a direct summand of $M$. If $M$ is a (dual-)Rickart module, so is $N$.
\end{cor}

\begin{prop}\label{complbaer}
	Let $\mathcal{L}$ be a complete modular lattice and $x\in\mathcal{L}$ be an element with complement $x'$. Let $\mathfrak{m}$ be a submonoid of $\End_{lin}(\mathcal{L})$ and let $\mathfrak{n}$ be a submonoid of $\End_{lin}([\mathbf{0},x])$ such that the extension given by Proposition \ref{exmorf} of every element $\psi\in\mathfrak{n}$ is in $\mathfrak{m}$. If $\mathcal{L}$ is $\mathfrak{m}$-Baer, then $[\mathbf{0},x]$ is a $\mathfrak{n}$-Baer lattice.
\end{prop}

\begin{proof}
	Consider a family of linear morphisms $\{\varphi_i:[\mathbf{0},x]\to[\mathbf{0},x]\mid \varphi_i\in\mathfrak{n}\}_I$ and let $k=\bigwedge_I\ker_{\varphi_i}$. By hypothesis, the extension  $\widehat{\varphi}_i:\mathcal{L}\to\mathcal{L}$  with $\ker_{\widehat{\varphi}_i}=\ker_{\varphi_i}\vee x'$ is in $\mathfrak{m}$ for all $i\in I$. Since $\_\wedge x:[x',\mathbf{1}]\to[\mathbf{0},x]$ is an isomorphism with inverse $\_\vee x'$,
	\[k=(k\vee x')\wedge x\leq\bigwedge_I(\ker_{\varphi_i}\vee x')\wedge x=k.\]
	Therefore, $k\vee x'=\bigwedge_I(\ker_{\varphi_i}\vee x')$. By hypothesis, $k\vee x'=\bigwedge_I(\ker_{\varphi_i}\vee x')$ has a complement in $\mathcal{L}$. Thus, $k$ has a complement in $[\mathbf{0},x]$.
\end{proof}

\begin{prop}\label{compldbaer}
	Let $\mathcal{L}$ be a complete modular lattice and $x\in\mathcal{L}$ be an element with complement $x'$. Let $\mathfrak{m}$ be a submonoid of $\End_{lin}(\mathcal{L})$ and let $\mathfrak{n}$ be a submonoid of $\End_{lin}([\mathbf{0},x])$ such that the extension given by Proposition \ref{exmorf} of every element $\psi\in\mathfrak{n}$ is in $\mathfrak{m}$. If $\mathcal{L}$ is dual-$\mathfrak{m}$-Baer, then $[\mathbf{0},x]$ is a dual-$\mathfrak{n}$-Baer lattice.
\end{prop}

\begin{cor}
	Let $M$ be a (dual-)Baer module. Then every direct summand $N$ of $M$ is also a (dual-)Baer module.
\end{cor}

We have to notice the short proof we present in Proposition \ref{complbaer} which gives as corollary \cite[Theorem 2.17]{rizvibaer}.

\begin{defn}
	A complete lattice $\mathcal{L}$ satisfies \emph{the} (resp. \emph{strong}) \emph{complement intersection property} CIP (resp. SCIP) if the infimum of any finite family (resp. any family) of complements in $\mathcal{L}$ is a complement.
\end{defn}

\begin{defn}
	A complete lattice $\mathcal{L}$ satisfies \emph{the} (resp. \emph{strong}) \emph{complement supremum property} CSP (resp. SCSP) if the supremum of any finite (resp. any family)  of complements in $\mathcal{L}$ is a complement.
\end{defn}

\begin{rem}
	For a module $M$, the lattice $\Lambda(M)$ satisfies (resp. SCIP) CIP is equivalent to say that $M$ satisfies the (resp. strong) summand intersection property (SIP) (resp. SSIP). The dual concept is called the (resp. strong) summand sum property (resp. SSSP) SSP. 
\end{rem}

\begin{prop}\label{baerricscip}
	Let $\mathcal{L}$ be a complete lattice and let $\mathfrak{m}$ be a submonoid of $\End_{lin}(\mathcal{L})$ containing all the projections. Then $\mathcal{L}$ is $\mathfrak{m}$-Rickart (resp. dual-$\mathfrak{m}$-Rickart) and satisfies SCIP (resp. SCSP) if and only if $\mathcal{L}$ is $\mathfrak{m}$-Baer (resp. dual-$\mathfrak{m}$-Baer).
\end{prop}

\begin{proof}
	Let $X\subseteq\mathfrak{m}$. Since $\mathcal{L}$ is $\mathfrak{m}$-Rickart, $\ker_\varphi$ is a complement in $\mathcal{L}$ for every $\varphi\in X$. By hypothesis, $\bigwedge_{\varphi\in X}\ker_\varphi$ is also a complement. Thus, $\mathcal{L}$ is $\mathfrak{m}$-Baer. Convesely,  let $\{c_i\}_I$ be a family of elements of $\mathcal{L}$ with complement and let $c_i'$ be a complement of $c_i$ for all $i\in I$. By hypothesis, $\pi_{c_i'}\in\mathfrak{m}$. Therefore $\bigwedge_Ic_i=\bigwedge_I\ker_{\pi_{c_i'}}$ is complemented in $\mathcal{L}$. This proves that $\mathcal{L}$ satisfies SCIP. It is clear that $\mathcal{L}$ is $\mathfrak{m}$-Rickart. 
\end{proof}

\begin{lemma}\label{lemmaret}
	Let $\mathcal{L}$ be a complete modular lattice. Let $a,b,c\in\mathcal{L}$ such that $a\wedge b=\mathbf{0}$, $(a\vee b)\wedge c=\mathbf{0}$. Then $a\wedge(b\vee c)=\mathbf{0}$.
\end{lemma}

\begin{prop}\label{riccipssp}
	Let $\mathcal{L}$ be a complete modular lattice and let $\mathfrak{m}$ be a submonoid of $\End_{lin}(\mathcal{L})$ containing all the projections. If $\mathcal{L}$ is $\mathfrak{m}$-Rickart (resp. dual-$\mathfrak{m}$-Rickart) then $\mathcal{L}$ satisfies CIP (resp. SSP).
\end{prop}

\begin{proof}
	Let $x$ and $y$ be two complements in $\mathcal{L}$. Let $x'$ and $y'$ the complements of $x$ and $y$ respectively. By hypothesis and Lemma \ref{kerpi} $(x\wedge y)\vee x'$ has a complement  in $\mathcal{L}$. That is, there exists $z\in\mathcal{L}$ such that $((x\wedge y)\vee x')\vee z=\mathbf{1}$ and $((x\wedge y)\vee x')\wedge z=\mathbf{0}$. By Lemma \ref{lemmaret} $x'\vee z$ is a complement of $x\wedge y$.
\end{proof}

\begin{cor}
	The following conditions are equivalent for an Artinian (resp. Noetherian) complete modular lattice $\mathcal{L}$ and a submonoid $\mathfrak{m}\subseteq\End_{lin}(\mathcal{L})$ containing all the projections.
	\begin{enumerate}[label=\emph{(\alph*)}]
		\item $\mathcal{L}$ is $\mathfrak{m}$-Rickart (resp. dual-$\mathfrak{m}$-Rickart).
		\item $\mathcal{L}$ is $\mathfrak{m}$-Baer (resp. dual-$\mathfrak{m}$-Baer).
	\end{enumerate}
\end{cor}

\begin{proof}
	Let $X\subseteq\mathfrak{m}$ and let $k=\bigwedge_{\varphi\in X}\ker_\varphi$. Since $\mathcal{L}$ is Artinian, so is $[k,\mathbf{1}]$. It follows from \cite[Proposition 1.3]{calugareanu2013lattice} that there exists a finite subset $F\subseteq X$ such that $k=\bigwedge_{\varphi\in F}\ker_\varphi$. Hence $k$ is a complement by Proposition \ref{riccipssp}. The converse is trivial.
\end{proof}

In \cite[Corollary 2.23]{leerickart}, it is proved that if the module $M\oplus M$ has SIP, then $M$ is a Rickart module. The lattice counterpart of this result is not true in general as the following example shows.

\begin{example}
	Consider the following lattice $\mathcal{L}$
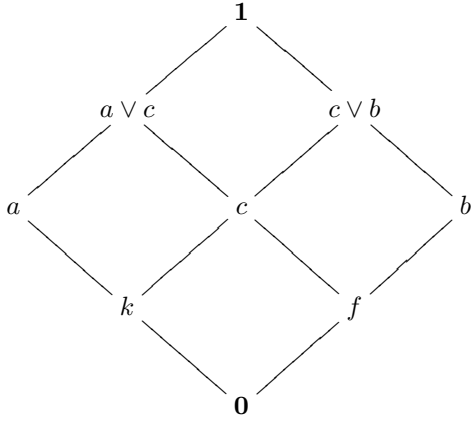
	\[\xymatrix{ & & \mathbf{1}\ar@{-}[dl] \ar@{-}[dr] & &  \\
	& a\vee c\ar@{-}[dl] \ar@{-}[dr] & & c\vee b\ar@{-}[dl] \ar@{-}[dr] & \\
	a\ar@{-}[dr] & & c\ar@{-}[dl] \ar@{-}[dr] & & b\ar@{-}[dl] \\
	 & k\ar@{-}[dr] & & f\ar@{-}[dl]  & \\
 	 & & \mathbf{0} & & }\]
  	Then $\mathcal{L}=[\mathbf{0},a\vee b]$ with $a\wedge b=\mathbf{0}$ and the lattices $[\mathbf{0},a]$ and $[\mathbf{0},b]$ are isomorphic. The elements with complement of $\mathcal{L}$ are $\{\mathbf{0},\mathbf{1},a,b\}$, therefore $\mathcal{L}$ has the CIP. Consider the linear morphism $\varphi:[\mathbf{0},a]\to [\mathbf{0},b]$ given by $\varphi(\mathbf{0})=\mathbf{0}$, $\varphi(k)=\mathbf{0}$ and $\varphi(a)=f$ (illustrated in Figure \ref{fig1}). Then, $\ker_\varphi=k$ but $k$ is not a complement in $[\mathbf{0},a]$.
\end{example}

\begin{defn}
	Let $\mathcal{L}$ be a modular lattice. Let $C(\mathcal{L})$ denote the set of complements in $\mathcal{L}$, that is,
	\[C(\mathcal{L})=\{a\in\mathcal{L}\mid a\;\text{has a complement}\}.\]
\end{defn}

\begin{cor}
	Let $\mathcal{L}$ be a complete modular lattice. If $\mathcal{L}$ is Baer or dual-Baer then $C(\mathcal{L})$ is a complete complemented lattice.
\end{cor}

\begin{prop}
	Let $\mathcal{L}$ be a complete modular lattice and suppose that $C(\mathcal{L})$ is a sublattice of $\mathcal{L}$. The following conditions are equivalent:
	\begin{enumerate}[label=\emph{(\alph*)}]
		\item $\pi_x(y)=\pi_y(x)=x\wedge y$ for all $x,y\in C(\mathcal{L})$.
		\item $C(\mathcal{L})$ is a Boolean algebra.
	\end{enumerate}
\end{prop}

\begin{proof}
	(a)$\Rightarrow$(b) Since $C(\mathcal{L})$ is a complemented modular lattice, we just have to prove that $C(\mathcal{L})$ is distributive. Let $x,y,z\in C(\mathcal{L})$. Then,
	\[x\wedge(y\vee z)=\pi_x(y\vee z)=\pi(y)\vee\pi_x(z)=(x\wedge y)\vee (x\wedge z).\]
	(b)$\Rightarrow$(a) Let $x,y\in C(\mathcal{L})$. Then 
	\[\pi_x(y)=(y\vee x')\wedge x=(y\wedge x)\vee(x\wedge x')=y\wedge x\]
	where $x'$ is a complement of $x$. Analogously, $\pi_y(x)=x\wedge y$.
\end{proof}

Note that if $M$ is an $R$-module, the set $C(\Lambda(M))$ consists of all direct summands of $M$.

\begin{cor}
	Let $M$ be an $R$-module and suppose that $C(\Lambda(M))$ is a sublattice of $\Lambda(M)$. The following conditions are equivalent:
	\begin{enumerate}[label=\emph{(\alph*)}]
		\item $\pi_N(L)=\pi_L(N)=N\cap L$ for all direct summands $N$ and $L$ of $M$.
		\item $C(\Lambda(M))$ is a Boolean algebra.
	\end{enumerate}
\end{cor}

\begin{cor}
    The following conditions are equivalent for a semisimple $R$-module $M$:
    \begin{enumerate}[label=\emph{(\alph*)}]
        \item $\pi_N(L)=\pi_L(N)=N\cap L$ for all submodules $N$ and $L$ of $M$.
        \item $M=\bigoplus_IS_i$ with $S_i\ncong S_j$ for every $i\neq j$.
    \end{enumerate}
\end{cor}

\begin{prop}\label{rickpix}
	Let $\mathcal{L}$ be a complete modular lattice. The following conditions are equivalent for $\varphi\in\End_{lin}(\mathcal{L})$:
	\begin{enumerate}[label=\emph{(\alph*)}]
		\item $\ker_\varphi$ is a complement in $\mathcal{L}$.
		\item There exists a complement $x\in\mathcal{L}$ such that $\varphi=\varphi\pi_x$ and $\ker_\varphi\wedge x=\mathbf{0}$.
	\end{enumerate}
\end{prop}

\begin{proof}
	(a)$\Rightarrow$(b) By hypothesis there exists $x\in\mathcal{L}$ such that $x$ is a complement of $\ker_\varphi$. Let $\ker_\varphi\leq y\leq\mathbf{1}$. By the modular property, $y=y\wedge\mathbf{1}=y\wedge(\ker_\varphi\vee x)=\ker_\varphi\vee(y\wedge x)$. Then $\varphi(y)=\varphi(y\wedge x)$. Let $a\in\mathcal{L}$. Then,
	\[\varphi\pi_x(a)=\varphi((a\vee\ker_\varphi)\wedge x)=\varphi(a\vee \ker_\varphi)=\varphi(a).\]
	
	(b)$\Rightarrow$(a) Let $x\in\mathcal{L}$ be a complement such that $\varphi=\varphi\pi_x$ and $x\wedge\ker_\varphi=\mathbf{0}$. Let $x'$ be a complement of $x$. Then, $\varphi(x')=\varphi\pi_x(x')=\varphi(\mathbf{0})=\mathbf{0}$. This implies that $x'\leq \ker_\varphi$. By the modularity, $\ker_\varphi=x'$.
\end{proof}

\begin{cor}
	Let $M$ be an $R$-module. The following conditions are equivalent for $f\in\End_R(M)$:
	\begin{enumerate}[label=\emph{(\alph*)}]
		\item $\Ker f$ is a direct summand of $M$.
		\item There exists a direct summand $N$ of $M$ such that $N\cap\Ker f=0$ and $f(L)=f\pi(L)$ for all $L\leq M$ where $\pi:M\to N$ is the canonical projection.
	\end{enumerate}
\end{cor}

\begin{cor}\label{rickpix2}
	The following conditions are equivalent for a complete modular lattice $\mathcal{L}$ and a submonoid $\mathfrak{m}\subseteq\End_{lin}(\mathcal{L})$ containing all the projections:
	\begin{enumerate}[label=\emph{(\alph*)}]
		\item $\mathcal{L}$ is $\mathfrak{m}$-Rickart.
		\item For every $\varphi\in\mathfrak{m}$ there exists a complement $x\in\mathcal{L}$ such that $\varphi=\varphi\pi_x$ and $x\wedge \ker_\varphi=\mathbf{0}$.
	\end{enumerate}
\end{cor}

\begin{cor}
    The following conditions are equivalent for an $R$-module $M$:
    \begin{enumerate}[label=\emph{(\alph*)}]
		\item $M$ is Rickart.
		\item For every $f\in\End_R(M)$ there exists a direct summand $N$ of $M$ such that $N\cap\Ker f=0$ and $f(L)=f\pi(L)$ for all $L\leq M$ where $\pi:M\to N$ is the canonical projection.
	\end{enumerate}
\end{cor}




\begin{lemma}\label{isolin}
	Let $\mathcal{L}$ be a complete modular lattice, $a,x\in\mathcal{L}$ and $\theta:[a,\mathbf{1}]\to[\mathbf{0},x]$ be an isomorphism of lattices. Then $\iota_x\theta\rho_a:\mathcal{L}\to \mathcal{L}$ is a linear morphism. 
\end{lemma}

\begin{proof}
	Define $\varphi:\mathcal{L}\to \mathcal{L}$ as $\varphi(y)=\iota_x\theta\rho_a(y)$. Note that $\varphi(y\vee a)=\theta(y\vee a\vee a)=\theta(y\vee a)=\varphi(y)$. Since $\theta:[a,\mathbf{1}]\to [\mathbf{0},x]$ is an isomorphism, it follows that $\ker_\varphi=a$.
\end{proof}

\begin{defn}
	Let $\mathcal{L}$ be a complete modular lattice, $a,x\in\mathcal{L}$ and $\mathfrak{m}\subseteq\End_{lin}(\mathcal{L})$ be a submonoid. It is said that $\mathcal{L}$ satisfies \emph{$\mathfrak{m}$-$D_2$ condition} if whenever there is an isomorphism $\theta:[a,\mathbf{1}]\overset{\cong}{\rightarrow}[\mathbf{0},x]$ with $x$ a complement in $\mathcal{L}$ and $\iota_x\theta\rho_a$ in $\mathfrak{m}$, implies that $a$ is a complement.
\end{defn}

Recall that an $R$-module $M$ satisfies the $(D_2)$ condition if whenever $N\cong M/L$ with $L,N\leq M$ and $N$ a direct summand of $M$ then $L$ is a direct summand of $M$ \cite[pp. 57]{mohamedcontinuous}. For example, every projective $R$-module satisfies ($D_2$).

\begin{lemma}
	The following conditions are equivalent for an $R$-module $M$:
	\begin{enumerate}[label=\emph{(\alph*)}]
		\item $M$ satisfies $(D_2)$ condition.
		\item $\Lambda(M)$ satisfies $\mathfrak{E}_M$-$D_2$ condition.
	\end{enumerate}
\end{lemma}

\begin{proof}
	(a)$\Rightarrow$(b) Let $N,L\in\Lambda(M)$ with $N$ a direct summand of $M$ and $\theta:[L,M]\to[0,N]$ an isomorphism of lattices such that $\iota_N\theta\rho_L\in\mathfrak{E}_M$. Then, there exists an endomorphism $f\in\End_R(M)$ such that $\iota_N\theta\rho_L=f_\ast$. Let $g:M/L\to N$ defined as $g(m+L)=f(m)$. Since $f_\ast(L)=0$, $g$ is well defined. Suppose $g(m+L)=0$. This implies that $f_\ast(Rm)=0$. Since $\theta$ is an isomorphism, $Rm\subseteq L$. Thus $g$ is inyective. Let $n\in N$. Consider $Rn$. Then, there exists $K$ with $L\leq K\leq M$ such that $\theta(K)=Rn$. Since $\rho_L$ is induced by the canonical projection $M\to M/L$, $f_\ast(K)=Rn$. This implies that there exists $k\in K$ such that $f(k)=n$. Hence $g$ is an isomorphism. By hypothesis, $L$ is a direct summand of $M$. Thus $\Lambda(M)$ satisfies $\mathfrak{E}_M$-$D_2$.
	
	(b)$\Rightarrow$(a) Let $N,L$ be submodules of $M$ with $N$ a direct summand and suppose that there exists an isomorphism $f:M/L\to N$. Then $f$ induces an isomorphism of lattices $\theta:[L,M]\to[0,N]$. Note that $\iota_N\theta\rho_L(K)=i_Nfp_L(K)$ where $i_N:N\to M$ and $p_L:M\to M/L$ are the canonical inclusion and the canonical projection respectively. Therefore $\iota_N\theta\rho_L=(i_Nfp_L)_\ast$, that is, $\iota_N\theta\rho_L\in\mathfrak{E}_M$. By hypothesis, $L$ is a complement in $\Lambda(M)$, that is, $L$ is a direct summand.
\end{proof}

\begin{defn}
	Let $\mathcal{L}$ be a complete modular lattice, $a,x,x'\in\mathcal{L}$ with $x'$ a complement of $x$. It is said that $\mathcal{L}$ satisfies \emph{$\mathfrak{m}$-$C_2$ condition} if whenever there is an isomorphism $\theta:[\mathbf{0},x]\overset{\cong}{\rightarrow}[\mathbf{0},a]$ and $\iota_a\theta(x\wedge\_)\rho_{x'}$ is in $\mathfrak{m}$, implies that $a$ is a complement.
\end{defn}

\begin{rem}
    In the case when $\mathfrak{m}=\End_{lin}(\mathcal{L})$ the above definition agrees with the condition $(C_2)$ given in \cite[Definition 1.1]{albu2016conditions}.
\end{rem}

Recall that an $R$-module $M$ satisfies the $(C_2)$ condition if whenever $N\cong L$ with $L,N\leq M$ and $N$ a direct summand of $M$ then $L$ is a direct summand of $M$ \cite[pp. 18]{mohamedcontinuous}. For example, every injective $R$-module satisfies ($C_2$).

\begin{lemma}
	The following conditions are equivalent for an $R$-module $M$.
	\begin{enumerate}[label=\emph{(\alph*)}]
		\item $M$ satisfies $(C_2)$ condition.
		\item $\Lambda(M)$ satisfies $\mathfrak{E}_M$-$C_2$ condition.
	\end{enumerate}
\end{lemma}

\begin{proof}
    (a)$\Rightarrow$(b) Let $N,L\in\Lambda(M)$ with $N$ a direct summand of $M$ with complement $N'$, and $\theta:[0,N]\to[0,L]$ an isomorphism of lattices such that $\iota_L\theta(N\cap\_)\rho_{N'}\in\mathfrak{E}_M$. Then, there exists an endomorphism $f\in\End_R(M)$ such that $\iota_L\theta(N\cap\_)\rho_{N'}=f_\ast$. Let $g:N\to L$ defined as $g(n)=f(n)$. Suppose $g(n)=0$. This implies that $f_\ast(Rn)=0=\theta(N\cap(Rn+N'))=\theta(Rn+(N\cap N')=\theta(Rn)$. Since $\theta$ is an isomorphism, $Rn=0$. Thus $g$ is inyective. Let $l\in L$. Consider $Rl$. Then, there exists $K$ with $0\leq K\leq N$ such that $\theta(K)=Rl$. Since $\rho_{N'}$ is induced by the canonical projection $M\to M/N'$, $f_\ast(K)=\theta(N\cap(K+N'))=\theta(K+(N\cap N'))=\theta(K)=Rl$. This implies that there exists $k\in K$ such that $f(n)=l$. Hence $g$ is an isomorphism. By hypothesis, $L$ is a direct summand of $M$. Thus $\Lambda(M)$ satisfies $\mathfrak{E}_M$-$C_2$.
    
    (b)$\Rightarrow$(a) Let $N,L$ be submodules of $M$ with $N$ a direct summand and suppose that there exists an isomorphism $f:N\to L$. Then $f$ induces an isomorphism of lattices $\theta:[0,N]\to[0,L]$. Note that $\iota_L\theta(N\cap\_)\rho_{N'}(K)=i_Lfp_N(K)$ where $i_L:L\to M$ and $p_N:M\to N$ are the canonical inclusion and the canonical projection respectively. Therefore $\iota_L\theta(N\cap\_)\rho_{N'}=(i_Lfp_N)_\ast$, that is, $\iota_L\theta(N\cap\_)\rho_{N'}\in\mathfrak{E}_M$. By hypothesis, $L$ is a complement in $\Lambda(M)$, that is, $L$ is a direct summand.
\end{proof}

\begin{prop}\label{ricd2}
	The following conditions are equivalent for a complete modular lattice $\mathcal{L}$ and a submonoid $\mathfrak{m}\subseteq\End_{lin}(\mathcal{L})$ containing all the projections:
	\begin{enumerate}[label=\emph{(\alph*)}]
		\item $\mathcal{L}$ is an $\mathfrak{m}$-Rickart Lattice.
		\item $\mathcal{L}$ satisfies $\mathfrak{m}$-$D_2$ condition and for every $\varphi\in\mathfrak{m}$ there exists an isomorphism  $\theta:[\mathbf{0},\varphi(\mathbf{1})]\to [\mathbf{0},x]$ with $x$ a complement in $\mathcal{L}$ such that $\iota_x\theta\varphi\in\mathfrak{m}$.
	\end{enumerate}
\end{prop}

\begin{proof}
	(a)$\Rightarrow$(b) Let $a\in\mathcal{L}$ such that there exists an isomorphism $\theta:[a,\mathbf{1}]\to [\mathbf{0},x]$ with $x$ a complement in $\mathcal{L}$ and $\iota_x\theta\rho_a$ in $\mathfrak{m}$. By Lemma \ref{isolin} and since $\mathcal{L}$ is $\mathfrak{m}$-Rickart, $a$ is a complement. On the other hand, let $\varphi\in \mathfrak{m}$. By hypothesis $\ker_\varphi$ is a complement in $\mathcal{L}$. Let $k'$ be a complement of $\ker_\varphi$. Then $[\mathbf{0},k']\cong [\ker_\varphi,\mathbf{1}]$ and $\overline{\varphi}:[\ker_\varphi,\mathbf{1}]\to[\mathbf{0},\varphi(\mathbf{1})]$ is an isomorphism. Define $\theta:[\mathbf{0},\varphi(\mathbf{1})]\to[\mathbf{0},k']$ as $\theta=(\_\wedge k')\overline{\varphi}^{-1}$. Then, $\iota_{k'}\theta\varphi(a)=(a\vee\ker_{\varphi})\wedge k'=\pi_{k'}(a)$. Since $\mathfrak{m}$ contains all the projections, $\iota_{k'}\theta\varphi\in\mathfrak{m}$.
	
	(b)$\Rightarrow$(a) Let $\varphi:\mathcal{L}\to \mathcal{L}$ be a linear morphism in $\mathfrak{m}$ with kernel $\ker_\varphi$. Then, there is an isomorphism $\overline{\varphi}:[\ker_\varphi,\mathbf{1}]\to[\mathbf{0}, \varphi(\mathbf{1})]$. By hypothesis, $[\mathbf{0},\varphi(\mathbf{1})]\cong_\theta[\mathbf{0},x]$ with $x$ a complement in $\mathcal{L}$ such that $\iota_x\theta\varphi\in \mathfrak{m}$. Consider the isomorphism $\theta\overline{\varphi}$. We claim that $\iota_x\theta\overline{\varphi}\rho_{\ker_\varphi}\in\mathfrak{m}$. For every $a\in\mathcal{L}$, we have that $\iota_x\theta\overline{\varphi}\rho_{\ker_\varphi}(a)=\iota_x\theta\varphi(a)$. Thus $\iota_x\theta\overline{\varphi}\rho_{\ker_\varphi}\in\mathfrak{m}$. By $\mathfrak{m}$-$D_2$ condition, $\ker_\varphi$ is complemented in $\mathcal{L}$. Thus, $\mathcal{L}$ is $\mathfrak{m}$-Rickart.
\end{proof}

\begin{prop}\label{dricc2}
	The following conditions are equivalent for a complete modular lattice $\mathcal{L}$ and a submonoid $\mathfrak{m}\subseteq\End_{lin}(\mathcal{L})$ containing all the projections:
	\begin{enumerate}[label=\emph{(\alph*)}]
		\item $\mathcal{L}$ is a dual-$\mathfrak{m}$-Rickart Lattice.
		\item $\mathcal{L}$ satisfies $\mathfrak{m}$-$C_2$ condition and for every $\varphi\in\mathfrak{m}$ there exists an isomorphism  $\theta:[\mathbf{0},x]\to [\mathbf{0},\varphi(\mathbf{1})]$ with $x$ a complement in $\mathcal{L}$ such that $\iota_{\varphi(\mathbf{1})}\theta(x\wedge\_)\rho_{x'}\in\mathfrak{m}$ where $x'$ is a complement of $x$.
	\end{enumerate}
\end{prop}

\begin{proof}
    (a)$\Rightarrow$(b) Let $a\in\mathcal{L}$ such that there exists an isomorphism $\theta:[\mathbf{0},x]\to [\mathbf{0},a]$ with $x$ a complement in $\mathcal{L}$ and $\iota_a\theta(x\wedge\_)\rho_{x'}$ in $\mathfrak{m}$ where $x'$ is a complement of $x$. By Lemma \ref{isolin} and since $\mathcal{L}$ is dual-$\mathfrak{m}$-Rickart, $a$ is a complement. On the other hand, let $\varphi\in \mathfrak{m}$. By hypothesis $\varphi(\mathbf{1})$ is a complement in $\mathcal{L}$. Let $b$ be a complement of $\varphi(\mathbf{1})$. Take $\theta=Id_{[\mathbf{0},\varphi(\mathbf{1})]}$. Thus $\iota_{\varphi(\mathbf{1})}\theta(\varphi(\mathbf{1})\wedge\_)\rho_{b}=\pi_{\varphi(\mathbf{1})}\in\mathfrak{m}$.
    
    (b)$\Rightarrow$(a) Let $\varphi:\mathcal{L}\to \mathcal{L}$ be a linear morphism in $\mathfrak{m}$. By hypothesis, $[\mathbf{0},x]\cong_\theta[\mathbf{0},\varphi(\mathbf{1})]$ with $x$ a complement in $\mathcal{L}$ such that $\iota_{\varphi(\mathbf{1})}\theta(x\wedge\_)\rho_{x'}\in \mathfrak{m}$ where $x'$ is a complement of $x$. It follows from  the $\mathfrak{m}$-$C_2$ condition, that $\varphi(\mathbf{1})$ is complemented in $\mathcal{L}$. Thus, $\mathcal{L}$ is dual-$\mathfrak{m}$-Rickart.
\end{proof}

\begin{cor}[{\cite[Proposition 2.11]{leerickart}}]
	The following conditions are equivalent for an $R$-module $M$.
	\begin{enumerate}[label=\emph{(\alph*)}]
		\item $M$ is a Rickart module.
		\item $M$ satisfies $(D_2)$ and for every $f\in\End_R(M)$ there is an isomorphism $g:f(M)\to N$ with $N$ a direct summand of $M$.
	\end{enumerate}
\end{cor}

\begin{cor}[{\cite[Proposition 2.21]{lee2011dual}}]
	The following conditions are equivalent for an $R$-module $M$.
	\begin{enumerate}[label=\emph{(\alph*)}]
		\item $M$ is a dual-Rickart module.
		\item $M$ satisfies $(C_2)$ and for every $f\in\End_R(M)$ there is an isomorphism $g:N\to f(M)$ with $N$ a direct summand of $M$.
	\end{enumerate}
\end{cor}

\begin{defn}
    Let $\mathcal{L}$ be a modular lattice and $\mathfrak{m}$ be a submonoid of $\End_{lin}(\mathcal{L})$. The lattice $\mathcal{L}$ is $k$-local $\mathfrak{m}$-retractable if for each $\varphi \in \mathfrak{m}$ and for each $b\leq \ker_\varphi$, there exists $\psi_{b}:\mathcal{L}\rightarrow [\mathbf{0},\ker_\varphi]$  such that $b\leq \psi_{b}(\mathbf{1})$ and $\iota_{\ker_\varphi}\psi_b\in\mathfrak{m}$.
\end{defn}

\begin{prop} 
Let $\mathcal{L}$ be a complete modular lattice and $\mathfrak{m}\subseteq\End_{lin}(\mathcal{L})$ be a submonoid containig all the projections. If $\mathcal{L}$ is $\mathfrak{m}$-Rickart, then $\mathcal{L}$ is $k$-local $\mathfrak{m}$-retractable.
\end{prop}

\begin{proof}  
Let  $\varphi:\mathcal{L}\rightarrow\mathcal{L}$ be a linear morphism. Since $\mathcal{L}$ is $\mathfrak{m}$-Rickart, we can consider the projection $\pi_{\ker_\varphi}:\mathcal{L}\to[\mathbf{0},\ker_\varphi]$ co-restricted. Then for $b\leq \ker_\varphi$,  we have that $b\leq\ker_\varphi=\pi_{\ker_\varphi}(\mathbf{1})$.
\end{proof}

\begin{defn}
	Let $\mathcal{L}$ be a complete modular lattice, $\mathfrak{m}$ be a submonoid of $\End_{lin}(\mathcal{L})$ and $x\in\mathcal{L}$. It is said that $x$ is $\mathfrak{m}$-$\mathcal{L}$-generated if there exists a family $\{\varphi_i:\mathcal{L}\to[\mathbf{0},x]\}_{i\in I}$ of linear morphisms such that $\iota_x\varphi_i\in\mathfrak{m}$ for all $i\in I$ and $x=\bigvee_{i\in I}\varphi_i(\mathbf{1})$.
\end{defn}

\begin{rem}
    It is clear that if $\mathcal{L}$ is $k$-local $\mathfrak{m}$-retractable, then $\ker_\varphi$ is $\mathfrak{m}$-$\mathcal{L}$-generated for all $\varphi\in\mathfrak{m}$. 
\end{rem}

\begin{defn}
	Let $\mathcal{L}$ be a complete modular lattice, $\mathfrak{m}$ be a submonoid of $\End_{lin}(\mathcal{L})$ and $x\in\mathcal{L}$. It is said that $x$ is \emph{$\mathfrak{m}$-$\mathcal{L}$-cogenerated} if there exists a family $\{\varphi_i:[x,\mathbf{1}]\to\mathcal{L}\}_{i\in I}$ of linear morphisms such that $\varphi_i\rho_x\in\mathfrak{m}$ for all $i\in I$ and $x=\bigwedge_{i\in I}\ker_{\varphi_i}$.
\end{defn}

\begin{prop}\label{kercompkergenann}
	Let $\mathcal{L}$ be a complete modular lattice and $\mathfrak{m}$ be a submonoid of $\End_{lin}(\mathcal{L})$ containing all the projections. The following conditions are equivalent for $\varphi\in\mathfrak{m}$:
	\begin{enumerate}[label=\emph{(\alph*)}]
		\item $\ker_\varphi$ has a complement in $\mathcal{L}$.
		\item $\ker_\varphi$ is $\mathfrak{m}$-$\mathcal{L}$-generated and $\annr(\varphi)=\varepsilon\mathfrak{m}$ for some idempotent $\varepsilon\in\mathfrak{m}$.
	\end{enumerate}
\end{prop}

\begin{proof}
	(a)$\Rightarrow$(b) Since $\ker_\varphi$ is a complement in $\mathcal{L}$, $\ker_\varphi=\pi_{\ker_\varphi}(\mathbf{1})$. Thus, $\ker_\varphi$ is $\mathfrak{m}$-$\mathcal{L}$-generated. By hypothesis, there exists $x\in\mathcal{L}$ such that $\mathbf{1}=\ker_\varphi\vee x$ and $\mathbf{0}=\ker_\varphi\wedge x$. We have that $\pi_{\ker_\varphi}(\mathbf{1})=\ker_\varphi$, so $\varphi\pi_{\ker_\varphi}(\mathbf{1})=\mathbf{0}$. This implies that $\pi_{\ker_\varphi}\in\annr(\varphi)$. Now, let $\psi\in\annr(\varphi)$, that is, $\varphi\psi=0$. Hence $\varphi\psi(\mathbf{1})=\mathbf{0}$. Therefore, $\psi(\mathbf{1})\leq\ker_\varphi=\pi_{\ker_\varphi}(\mathbf{1})$. Let $a\in\mathcal{L}$. Then $\psi(a)=(\psi(a)\vee x)\wedge\ker_\varphi=\pi_{\ker_\varphi}\psi(a)$. Thus $\pi_{\ker_\varphi}\psi=\psi$. This implies that $\annr(\varphi)=\pi_{\ker_\varphi}\mathfrak{m}$.
	
	(b)$\Rightarrow$(a) By hypothesis there exists $\varepsilon\in\mathfrak{m}$ idempotent such that $\annr(\varphi)=\varepsilon\mathfrak{m}$. Then $\varphi\varepsilon(\mathbf{1})=\mathbf{0}$. Thus $\varepsilon(\mathbf{1})\leq\ker_\varphi$. On the other hand, there exists linear morphisms $\psi_i:\mathcal{L}\to[\mathbf{0},\ker_\varphi]$ ($i\in I$) such that $\bigvee_{i\in I}\psi_i(\mathbf{1})=\ker_\varphi$. This implies that $\varphi\psi_i=0$ for all $i\in I$ and so $\psi_i\in\annr(\varphi)$ for all $i\in I$. Therefore $\varepsilon\psi_i=\psi_i$ for all $i\in I$. Then
	\[\ker_\varphi=\bigvee_{i\in I}\psi_i(\mathbf{1})=\bigvee_{i\in I}\varepsilon\psi_i(\mathbf{1})\leq\varepsilon(\mathbf{1}).\]
	Thus, $\varepsilon(\mathbf{1})=\ker_\varphi$. Since $\varepsilon$ is idempotent, $\ker_\varphi$ is a complement in $\mathcal{L}$.
\end{proof}

\begin{prop}\label{imcompintkercogen}
	Let $\mathcal{L}$ be a complete modular lattice and $\mathfrak{m}$ be a submonoid of $\End_{lin}(\mathcal{L})$ containing all the projections. The following conditions are equivalent for $\varphi\in\mathfrak{m}$:
	\begin{enumerate}[label=\emph{(\alph*)}]
		\item $\varphi(\mathbf{1})$ has a complement in $\mathcal{L}$.
		\item $\varphi(\mathbf{1})$ is $\mathfrak{m}$-$\mathcal{L}$-cogenerated and $\annl(\varphi)=\mathfrak{m}\varepsilon$ for some idempotent $\varepsilon\in\mathfrak{m}$.
	\end{enumerate}
\end{prop}

\begin{proof}
	(a)$\Rightarrow$(b) Since $\varphi(\mathbf{1})$ is complemented in $\mathcal{L}$, there exists $x\in\mathcal{L}$ such that $\mathbf{1}=\varphi(\mathbf{1})\vee x$ and $\mathbf{0}=\varphi(\mathbf{1})\wedge x$. Then $\varphi(\mathbf{1})=\ker_{\iota_x(x\wedge\_)}$. Thus, $\varphi(\mathbf{1})$ is $\mathfrak{m}$-$\mathcal{L}$-cogenerated. We have that $\pi_{x}\varphi(\mathbf{1})=\mathbf{0}$. This implies that $\pi_{x}\in\annl(\varphi)$. Now, let $\psi\in\annl(\varphi)$, that is, $\psi\varphi=0$. Hence $\psi\varphi(\mathbf{1})=\mathbf{0}$. Therefore, $\varphi(\mathbf{1})\leq\ker_\psi$. Let $a\in\mathcal{L}$. Then $\psi\pi_x(a)=\psi((a\vee \varphi(\mathbf{1}))\wedge x)\leq\psi((a\vee \varphi(\mathbf{1})))\wedge \psi(x)=\psi(a)\wedge\psi(x)=\psi(a)\wedge\psi(\mathbf{1})=\psi(a)$. On the other hand, $a\leq\pi_{\varphi(\mathbf{1})}(a)\vee\pi_x(a)$. Hence $\psi(a)\leq\psi(\pi_{\varphi(\mathbf{1})}(a)\vee\pi_x(a))=\psi(\pi_{\varphi(\mathbf{1})}(a))\vee\psi(\pi_x(a))=\psi\pi_x(a)$. Thus $\psi\pi_{x}=\psi$. This implies that $\annl(\varphi)=\mathfrak{m}\pi_{x}$.
	
	(b)$\Rightarrow$(a) By hypothesis there exists $\varepsilon\in\mathfrak{m}$ idempotent such that $\annl(\varphi)=\mathfrak{m}\varepsilon$. Then $\varepsilon\varphi(\mathbf{1})=\mathbf{0}$. Thus $\varphi(\mathbf{1})\leq\ker_\varepsilon$. On the other hand, there exists linear morphisms $\psi_i:[\varphi(\mathbf{1}),\mathbf{1}]\to\mathcal{L}$ ($i\in I$) such that $\bigwedge_{i\in I}\ker_{\psi_i}=\varphi(\mathbf{1})$ and $\psi_i\rho_{\varphi(\mathbf{1})}\in\mathfrak{m}$. This implies that $\psi_i\rho_{\varphi(\mathbf{1})}\varphi=0$ for all $i\in I$ and so $\psi_i\rho_{\varphi(\mathbf{1})}\in\annl(\varphi)$ for all $i\in I$. Therefore $\psi_i\rho_{\varphi(\mathbf{1})}\varepsilon=\psi_i\rho_{\varphi(\mathbf{1})}$ for all $i\in I$. Then, for each $i\in I$,
	\[\psi_i(\ker_\varepsilon)=\psi_i(\ker_\varepsilon\vee\varphi(\mathbf{1}))=\psi_i\rho_{\varphi(\mathbf{1})}(\ker_{\varepsilon})=\psi_i\rho_{\varphi(\mathbf{1})}\varepsilon(\ker_\varepsilon)=\mathbf{0}.\]
	This implies that $\ker_\varepsilon\leq\bigwedge_{i\in I}\ker_{\psi_i}=\varphi(\mathbf{1})$. Thus, $\varphi(\mathbf{1})=\ker_\varepsilon$. Since $\varepsilon$ is idempotent, $\varphi(\mathbf{1})$ is complemented in $\mathcal{L}$.
\end{proof}

\begin{defn}
	A monoid $E$ with zero element is called \emph{right Rickart} if for every $\varphi\in E$ there exists an idempotent element $\varepsilon\in E$ such that $\annr(\varphi)=\{\psi\in E\mid \varphi\psi=0\}=\varepsilon E$. A \emph{left Rickart} monoid is defined similarly.
\end{defn}

It is clear that if a ring $R$ is right Rickart if and only if $R$ as monoid with the multiplication is right Rickart in the sense of the above definition. In \cite{leerickart} it is proved that, if $M$ is a Rickart module, then $\End_R(M)$ is a right Rickart ring. The next results explore the lattice-counterpart of this fact.

\begin{cor}\label{ricendoric}
	The following conditions are equivalent for a complete modular lattice $\mathcal{L}$ and a submonoid $\mathfrak{m}\subseteq\End_{lin}(\mathcal{L})$ containing all the projections:
	\begin{enumerate}[label=\emph{(\alph*)}]
		\item $\mathcal{L}$ is $\mathfrak{m}$-Rickart.
		\item The monoid $\mathfrak{m}$ is right Rickart and $\mathcal{L}$ is $k$-local $\mathfrak{m}$-retractable.
		\item The monoid $\mathfrak{m}$ is right Rickart and $\ker_\varphi$ is $\mathfrak{m}$-$\mathcal{L}$-generated for all $\varphi\in\mathfrak{m}$.
	\end{enumerate}
\end{cor}

\begin{cor}\label{dricendodric}
	The following conditions are equivalent for a complete modular lattice $\mathcal{L}$ and a submonoid $\mathfrak{m}\subseteq\End_{lin}(\mathcal{L})$ containing all the projections:
	\begin{enumerate}[label=\emph{(\alph*)}]
		\item $\mathcal{L}$ is dual-$\mathfrak{m}$-Rickart.
		\item The monoid $\mathfrak{m}$ is left Rickart and $\varphi(\mathbf{1})$ is $\mathfrak{m}$-$\mathcal{L}$-cogenerated for all $\varphi\in\mathfrak{m}$.
	\end{enumerate}
\end{cor}

\begin{cor}[{Theorem 3.9, \cite{leerickart}}]
	The following conditions are equivalent for an $R$-module $M$:
	\begin{enumerate}[label=\emph{(\alph*)}]
		\item $M$ is Rickart.
		\item The ring $\End_{R}(M)$ is right Rickart and $M$ is $k$-local-retractable.
		\item The ring $\End_{R}(M)$ is right Rickart and $\Ker\varphi$ is $M$-generated for all $\varphi\in\End_{R}(M)$.
	\end{enumerate}
\end{cor}

\begin{cor}
	The following conditions are equivalent for an $R$-module $M$:
	\begin{enumerate}[label=\emph{(\alph*)}]
		\item $M$ is dual-Rickart.
		\item The ring $\End_{R}(M)$ is left Rickart and $\Coker\varphi$ is $M$-cogenerated for all $\varphi\in\End_{R}(M)$.
	\end{enumerate}
\end{cor}

\begin{defn}
	A monoid $E$ with zero element is called \emph{right Baer} if for every $X\subseteq E$ there exists an idempotent element $\varepsilon\in E$ such that $\annr(X)=\{\psi\in E\mid \varphi\psi=0\;\forall \varphi\in X\}=\varepsilon E$. A \emph{left Baer} monoid is defined similarly.
\end{defn}

It is clear that if a ring $R$ is Baer if and only if $R$ as monoid with the multiplication is Baer in the sense of the above definition, moreover the Baer condition on a ring is left-right symmetric. Below we prove that for the case of the monoid $\End_{lin}(\mathcal{L})$ with $\mathcal{L}$ a complete modular lattice, the right and left Baer conditions are equivalent (Compare).

\begin{lemma}
    The following conditions are equivalent for a complete modular lattice $\mathcal{L}$ and a submonoid $\mathfrak{m}\subseteq\End_{lin}(\mathcal{L})$ containing all the projections:
    \begin{enumerate}[label=\emph{(\alph*)}]
        \item For every $X\subseteq \mathfrak{m}$, $\annl(X)=\{\psi\in \mathfrak{m}\mid \psi\varphi=0\;\forall \varphi\in X\}=\mathfrak{m}\zeta$ for some idempotent $\zeta\in \mathfrak{m}$.
        \item For every $Y\subseteq \mathfrak{m}$, $\annr(Y)=\{\psi\in \mathfrak{m}\mid \varphi\psi=0\;\forall \varphi\in X\}=\varepsilon \mathfrak{m}$ for some idempotent $\varepsilon\in \mathfrak{m}$. 
    \end{enumerate}
\end{lemma}

\begin{proof}
    Since the conditions are symmetric, we only prove (a)$\Rightarrow$(b). Let $Y\subseteq \mathfrak{m}$. Then $\annr(\annl(\annr(Y)))=\annr(Y)$. By hypothesis, there exists an idempotent $\zeta\in \mathfrak{m}$ such that $\annl(\annr(Y))=\mathfrak{m}\zeta$. By Proposition \ref{idemcomp}, $\ker_\zeta$ is a complement in $\mathcal{L}$. Consider $\pi_{\ker_\zeta}\in \mathfrak{m}$. We claim that $\annr(Y)=\annr(\mathfrak{m}\zeta)=\pi_{\ker_\zeta}\mathfrak{m}$. Since $\zeta\pi_{\ker_\zeta}=0$, then $\pi_{\ker_\zeta}\mathfrak{m}\subseteq\annr(Y)$. Now, let $\psi\in\annr(Y)=\annr(\mathfrak{m}\zeta)$. Then $\zeta\psi=0$. This implies that $\psi(\mathbf{1})\leq\ker_\zeta$. Hence $\pi_{\ker_\zeta}(\psi(a))=\psi(a)$ for all $a\in \mathcal{L}$. Thus, $\psi\in \pi_{\ker_\zeta}\mathfrak{m}$. This proves the claim.
\end{proof}

In \cite[Theorem 4.1]{rizvibaer} it is proved that, if $M$ is a Baer module, then $\End_R(M)$ is a Baer ring. The next results explore the lattice-counterpart of this fact.

\begin{prop}\label{baercar}
	Let $\mathcal{L}$ be a complete modular lattice and $\mathfrak{m}$ be a submonoid of $\End_{lin}(\mathcal{L})$ containing all the projections. The following conditions are equivalent:
	\begin{enumerate}[label=\emph{(\alph*)}]
		\item $\mathcal{L}$ is $\mathfrak{m}$-Baer.
		\item For every $a\in\mathcal{L}$ there exists an idempotent $\varepsilon\in\mathfrak{m}$ such that  $\{\varphi\in\mathfrak{m}\mid \varphi(a)=\mathbf{0}\}=\mathfrak{m}\varepsilon$.
		\item The monoid $\mathfrak{m}$ is Baer and $\bigwedge_{\varphi\in X}\ker_\varphi$ is $\mathfrak{m}$-$\mathcal{L}$-generated for all $X\subseteq\mathfrak{m}$.
	\end{enumerate}
\end{prop}

\begin{proof}
	(a)$\Rightarrow$(b) Let $a\in\mathcal{L}$. Consider $X_a=\{\varphi\in\mathfrak{m}\mid \varphi(a)=\mathbf{0}\}$ and $k=\bigwedge_{\varphi\in X_a}\ker_\varphi$. Note that $X_a=\{\varphi\in\mathfrak{m}\mid \varphi(k)=\mathbf{0}\}$. By hypothesis, $k$ has a complement $k'$ in $\mathcal{L}$. Then $\pi_{k'}\in X_a$. If $\psi\in\mathfrak{m}$, then $\psi\pi_{k'}(k)=\mathbf{0}$. Thus, $\mathfrak{m}\pi_{k'}\subseteq X_a$. Now, let $\varphi\in X_a$ and $x\in\mathcal{L}$. We have that $\varphi(k)=\mathbf{0}$, that is, $k\leq\ker_\varphi$. Also, $\varphi(\mathbf{1})=\varphi(k\vee k')=\varphi(k')$. Then 
	\[\varphi\pi_{k'}(x)=\varphi((x\vee k)\wedge k')\leq\varphi(x\vee k)\wedge\varphi(k')=\varphi(x)\wedge\varphi(k')=\varphi(x).\]
	On the other hand, it follows from Proposition \ref{fipi} that $x\leq\pi_{k}(x)\vee\pi_{k'}(x)$. Therefore, $\varphi(x)\leq\varphi(\pi_{k}(x)\vee\pi_{k'}(x))=\varphi(\pi_{k}(x))\vee\varphi(\pi_{k'}(x))=\varphi\pi_{k'}(x)$. Thus, $\varphi=\varphi\pi_{k'}$ and so $X_a=\mathfrak{m}\varepsilon$ with $\varepsilon=\pi_{k'}$.
	
	(b)$\Rightarrow$(c) Let $X\subseteq \mathfrak{m}$ and $k=\bigwedge_{\varphi\in X}\ker_\varphi$. By hypothesis there exists an idempotent $\varepsilon\in\mathfrak{m}$ such that $X_k=\{\psi\in\mathfrak{m}\mid \psi(k)=\mathbf{0}\}=\mathfrak{m}\varepsilon$. Suppose $\varepsilon(a)=\mathbf{0}$ for some $a\in\mathcal{L}$. Since $X\subseteq X_k$, $\varphi(a)=\varphi\varepsilon(a)=\mathbf{0}$. Hence $a\leq k$. Thus, $k=\ker_\varepsilon$. This proves that $k$ is $\mathfrak{m}$-$\mathcal{L}$-generated.	Consider $\pi_k$. Since $\varepsilon\pi_k=0$, $\pi_k\mathfrak{m}\subseteq\annr(X)$. Now, let $\xi\in\annr(X)$, that is, $\varphi\xi=0$ for all $\varphi\in X$. This implies that $\xi(\mathbf{1})\leq k$ and so $\varepsilon\xi(\mathbf{1})\leq\varepsilon(k)=\mathbf{0}$. Therefore, $\xi(\mathbf{1})\leq k$. It follows that $\xi=\xi\pi_k$. Thus, $\annr(X)=\pi_k\mathfrak{m}$.
	
	(c)$\Rightarrow$(a) A small modification of the proof (b)$\Rightarrow$(a) of Proposition \ref{kercompkergenann} works here.
	
	Let $X\subseteq\mathfrak{m}$. By hypothesis there exists $\varepsilon\in\mathfrak{m}$ idempotent such that $\annr(X)=\varepsilon\mathfrak{m}$. Then $\varphi\varepsilon(\mathbf{1})=\mathbf{0}$ for all $\varphi\in X$. Thus $\varepsilon(\mathbf{1})\leq\ker_\varphi$ for all $\varphi\in X$. On the other hand, there exists linear morphisms $\psi_i:\mathcal{L}\to[\mathbf{0},\bigwedge_{\varphi\in X}\ker_\varphi]$ ($i\in I$) such that $\bigvee_{i\in I}\psi_i(\mathbf{1})=\bigwedge_{\varphi\in X}\ker_\varphi$. This implies that $\varphi\psi_i=0$ for all $i\in I$ and for all $\varphi\in X$, and so $\psi_i\in\annr(X)$ for all $i\in I$. Therefore $\varepsilon\psi_i=\psi_i$ for all $i\in I$. Then
	\[\bigwedge_{\varphi\in X}\ker_\varphi=\bigvee_{i\in I}\psi_i(\mathbf{1})=\bigvee_{i\in I}\varepsilon\psi_i(\mathbf{1})\leq\varepsilon(\mathbf{1}).\]
	Thus, $\varepsilon(\mathbf{1})=\bigwedge_{\varphi\in X}\ker_\varphi$. Since $\varepsilon$ is idempotent, $\bigwedge_{\varphi\in X}\ker_\varphi$ is a complement in $\mathcal{L}$.
\end{proof}

\begin{cor}
    The following conditions are equivalent for an $R$-module $M$ with $S=\End_R(M)$:
    \begin{enumerate}[label=\emph{(\alph*)}]
        \item $M$ is Baer.
        \item For every $N\leq M$ there exists an idempotent $e\in S$ such that $\{f\in S\mid f(N)=0\}=Se$
        \item $S$ is a Baer ring and $\bigwedge_{f\in X}\Ker f$ is $M$-generated for all $X\subseteq S$.
    \end{enumerate}
\end{cor}

\begin{prop}\label{dbaercar}
	Let $\mathcal{L}$ be a complete modular lattice and $\mathfrak{m}$ be a submonoid of $\End_{lin}(\mathcal{L})$ containing all the projections. The following conditions are equivalent:
	\begin{enumerate}[label=\emph{(\alph*)}]
		\item $\mathcal{L}$ is dual-$\mathfrak{m}$-Baer.
		\item For every $a\in\mathcal{L}$ there exists an idempotent $\varepsilon\in\mathfrak{m}$ such that  $\{\varphi\in\mathfrak{m}\mid \varphi(\mathbf{1})\leq a\}=\varepsilon\mathfrak{m}$.
		\item The monoid $\mathfrak{m}$ is Baer and $\bigvee_{\varphi\in X}\varphi(\mathbf{1})$ is $\mathfrak{m}$-$\mathcal{L}$-cogenerated for all $X\subseteq\mathfrak{m}$.
	\end{enumerate}
\end{prop}

\begin{proof}
	(a)$\Rightarrow$(b) Let $a\in\mathcal{L}$. Consider $D_a=\{\varphi\in\mathfrak{m}\mid \varphi(\mathbf{1})\leq a\}$ and $d=\bigvee_{\varphi\in D_a}\varphi(\mathbf{1})$. Note that $D_a=\{\varphi\in\mathfrak{m}\mid \varphi(\mathbf{1})\leq d\}$. By hypothesis, $d$ has a complement $d'$ in $\mathcal{L}$. Then $\pi_{d}\in D_a$. If $\psi\in\mathfrak{m}$, then $\pi_{d}\psi(\mathbf{1})\leq d$. Thus, $\pi_{d}\mathfrak{m}\subseteq D_a$. Now, let $\varphi\in D_a$ and $x\in\mathcal{L}$. We have that $\varphi(x)\leq d$. Then $\pi_{d}\varphi(x)=(\varphi(x)\vee d')\wedge d=\varphi(x)$ by modularity.
	Thus, $D_a=\varepsilon\mathfrak{m}$ with $\varepsilon=\pi_{d}$.
	
	(b)$\Rightarrow$(a)(c) Let $X\subseteq\mathfrak{m}$ and $d=\bigvee_{\varphi\in X}\varphi(\mathbf{1})$. Consider $D_d=\{\varphi\in\mathfrak{m}\mid \varphi(\mathbf{1})\leq d\}$. By hypothesis, $D_d=\varepsilon\mathfrak{m}$ for some idempotent $\varepsilon$. It follows that $\varepsilon(\mathbf{1})\leq d$. On the other hand, if $\varphi\in D_d$, then $\varphi=\varepsilon\varphi$. Therefore $\varphi(\mathbf{1})=\varepsilon\varphi(\mathbf{1})\leq\varepsilon(\mathbf{1})$. This implies $d=\bigvee_{\varphi\in X}\varphi(\mathbf{1})\leq\varepsilon(\mathbf{1})$. Thus $\varepsilon(\mathbf{1})=d$ and by Proposition \ref{idemcomp}, $d$ is complemented in $\mathcal{L}$. This proves (a). Now, consider the linear morphism $[d,\mathbf{1}]\to \mathcal{L}$ given by the composition of the isomorphism $\ker_\varepsilon\wedge\_:[d,\mathbf{1}]\to[\mathbf{0},\ker_\varepsilon]$ and the canonical inclusion $\iota_{\ker_\varepsilon}$. Then $\iota_{\ker_\varepsilon}(\ker_\varepsilon\wedge\_)\rho_d=\pi_{\ker_\varepsilon}\in\mathfrak{m}$. This implies that $d$ is $\mathfrak{m}$-$\mathcal{L}$-cogenerated. Since $\varphi=\varepsilon\varphi$, $\pi_{\ker_\varepsilon}\varphi=0$. Then $\mathfrak{m}\pi_{\ker_\varepsilon}\subseteq\annl(X)$. Let $\psi\in\annl(X)$, that is, $\psi\varphi=0$ for all $\varphi\in X$. It follows that $\psi\varphi(\mathbf{1})=\mathbf{0}$ for all $\varphi\in X$ and so $\psi(d)=\mathbf{0}$. Hence $\psi\varepsilon=0$. By Proposition \ref{fipi}.(1), $\psi(a)\leq\psi\pi_{\ker_\varepsilon}(a)$ for all $a\in\mathcal{L}$. On the other hand, 
	\[\psi\pi_{\ker_\varepsilon}(a)=\psi((a\vee d)\wedge\ker_\varepsilon)\leq\psi(a\vee d)\wedge\psi(\ker_\varepsilon)\leq\psi(a)\vee\psi(d)=\psi(a).\]
	Thus, $\psi=\psi\pi_{\ker_\varepsilon}$. This implies that $\annl(X)=\mathfrak{m}\pi_{\ker_\varepsilon}$.
	
	(c)$\Rightarrow$(a) It follows from a small modification on the proof (b)$\Rightarrow$(a) of Proposition \ref{imcompintkercogen}.

	Let $X\subseteq\mathfrak{m}$ and $d=\bigvee_{\varphi\in X}\varphi(\mathbf{1})$. By hypothesis there exists $\varepsilon\in\mathfrak{m}$ idempotent such that $\annl(X)=\mathfrak{m}\varepsilon$. Then $\varepsilon\varphi(\mathbf{1})=\mathbf{0}$ for all $\varphi\in X$. Thus $\varphi(\mathbf{1})\leq\ker_\varepsilon$ for all $\varphi\in X$. On the other hand, there exists linear morphisms $\psi_i:[d,\mathbf{1}]\to\mathcal{L}$ ($i\in I$) such that $\bigwedge_{i\in I}\ker_{\psi_i}=d$ and $\psi_i\rho_d\in\mathfrak{m}$. This implies that $\psi_i\rho_d\varphi=0$ for all $i\in I$ and for all $\varphi\in X$, and so $\psi_i\rho_d\in\annl(\varphi)$ for all $i\in I$. Therefore $\psi_i\rho_d\varepsilon=\psi_i\rho_d$ for all $i\in I$. Then, for each $i\in I$,
	\[\psi_i(\ker_\varepsilon)=\psi_i(\ker_\varepsilon\vee d)=\psi_i\rho_d(\ker_{\varepsilon})=\psi_i\rho_d\varepsilon(\ker_\varepsilon)=\mathbf{0}.\]
	This implies that $\ker_\varepsilon\leq\bigwedge_{i\in I}\ker_{\psi_i}=d$. Thus, $d=\ker_\varepsilon$. Since $\varepsilon$ is idempotent, $d$ is complemented in $\mathcal{L}$.
\end{proof}

\begin{cor}
    The following conditions are equivalent for an $R$-module $M$ with $S=\End_R(M)$:
    \begin{enumerate}[label=\emph{(\alph*)}]
        \item $M$ is dual-Baer.
        \item For every $N\leq M$ there exists an idempotent $e\in S$ such that $\{f\in S\mid f(M)\leq N\}=eS$
        \item $S$ is a Baer ring and the factor module $M/\bigvee_{f\in X}f(M)$ is $M$-cogenerated for all $X\subseteq S$.
    \end{enumerate}
\end{cor}

\begin{cor}
The following conditions are equivalent for an $R$-module $M$ with $S=\End_R(M)$:
    \begin{enumerate}[label=\emph{(\alph*)}]
        \item $M$ is Baer and dual-Baer.
        \item $S$ is a Baer ring, $\bigwedge_{f\in X}\Ker f$ is $M$-generated and $M/\bigvee_{f\in X}f(M)$ is $M$-cogenerated for all $X\subseteq S$.
    \end{enumerate}
    In this case the set of direct summnads of $M$ is a complete complemented sublattice of $\Lambda(M)$.
\end{cor}

The following propositions give conditions in order to the concepts of Rickart and Baer (resp. dual-Rickart and dual-Baer) coincide. As corollaries we get \cite[Theorem 4.2]{lee2011dual} and \cite[Theorem 4.5]{leerickart}

\begin{prop}
	Let $\mathcal{L}$ be a complete modular lattice such that $C(\mathcal{L})$ satisfies ACC and $\mathfrak{m}$ be a submonoid of $\End_{lin}(\mathcal{L})$ containing all the projections. The following conditions are equivalent:
	\begin{enumerate}[label=\emph{(\alph*)}]
		\item $\mathcal{L}$ is $\mathfrak{m}$-Rickart.
		\item $\mathcal{L}$ is $\mathfrak{m}$-Baer.
	\end{enumerate}
\end{prop}

\begin{proof}
	Suppose $\mathcal{L}$ is $\mathfrak{m}$-Rickart. For $a\in\mathcal{L}$ let $X_a$ denote the set $\{\varphi\in\mathfrak{m}\mid\varphi(a)=\mathbf{0}\}$. If $\varphi\in X_a$, then $a\leq\ker_\varphi$. Let $k'$ be a complement of $\ker_\varphi$ in $\mathcal{L}$. Then $\pi_{k'}(a)=\mathbf{0}$, that is, $\pi_{k'}\in X_a$ with $k'\in C(\mathcal{L})$. Therefore, if $0\neq X_a$ then there always exists $\mathbf{0}\neq x\in C(\mathcal{L})$ such that $\pi_x\in X_a$. Let $\Gamma=\{x\in C(\mathcal{L})\mid \pi_x\in X_a\}$. Then $\Gamma\neq\emptyset$ and by hypothesis $\Gamma$ has maximal elements. Let $k\in\Gamma$ be a maximal and $k'$ be a complement of $k$. We claim that $X_a\cap X_k=0$. Suppose $0\neq X_a\cap X_k=X_{a\vee k}$. By the comment above, there exists $\mathbf{0}\neq x\in C(\mathcal{L})$ such that $\pi_x\in X_{a\vee k}$. Let $x'$ be a complement of $x$ in $\mathcal{L}$. Hence $k\leq x'$. Note that $a\leq k'\wedge x'$. Since $\mathcal{L}$ is Rickart, $x'\wedge k'\in C(\mathcal{L})$ and following the proof of Proposition \ref{riccipssp}, a complement for $k'\wedge x'$ is $k\vee x$. It follows that $k< k\vee x$ and $k\vee x\in\Gamma$ which is a contradiction. Thus, $X_a\cap X_k=0$. Now, since $\pi_k\in X_a$, $a\leq \pi_{k}(a)\vee\pi_{k'}(a)=\pi_{k'}(a)\leq k'$. This implies that $\pi_{k'}(a)=a$. Let $\varphi\in X_a$. Then $\varphi\pi_{k'}(a)=\varphi(a)=\mathbf{0}$. Hence $\varphi\pi_{k'}\in X_a\cap X_k=0$. Thus, $k'\leq \ker_\varphi$ for all $\varphi\in X_a$. Then  $k'=\bigwedge_{\varphi\in X_a}\ker_\varphi$ and so $\mathcal{L}$ is Baer by Proposition \ref{baercar}
\end{proof}

\begin{prop}
	Let $\mathcal{L}$ be a complete modular lattice such that $C(\mathcal{L})$ satisfies ACC and $\mathfrak{m}$ be a submonoid of $\End_{lin}(\mathcal{L})$ containing all the projections. The following conditions are equivalent:
	\begin{enumerate}[label=\emph{(\alph*)}]
		\item $\mathcal{L}$ is dual-$\mathfrak{m}$-Rickart.
		\item $\mathcal{L}$ is dual-$\mathfrak{m}$-Baer.
	\end{enumerate}
\end{prop}

\begin{proof}
	Suppose $\mathcal{L}$ is dual-$\mathfrak{m}$-Rickart. Let $X\subseteq\mathfrak{m}$ and $a=\bigvee_{\varphi\in X}\varphi(\mathbf{1})$. Consider $\Gamma=\{x\in C(\mathcal{L})\mid x\leq a\}$. Then $\Gamma\neq\emptyset$ and by hypothesis $\Gamma$ has maximal elements. Let $\ell\in\Gamma$ be a maximal. Then $\ell\leq a$. Suppose that $\ell< a$. This implies that there exists $\psi\in\mathfrak{m}$ such that $\psi(\mathbf{1})\leq a$ and $\psi(\mathbf{1})\nleq \ell$. By Proposition \ref{riccipssp}, $\ell\vee\psi(\mathbf{1})\in C(\mathcal{L})$ and $\ell<\ell\vee\psi(\mathbf{1})\leq a$. This is a contradiction. Then  $\ell=a=\bigvee_{\varphi\in X}\varphi(\mathbf{1})$ and so $\mathcal{L}$ is dual-$\mathfrak{m}$-Baer.
\end{proof}

\section{(Dual-)Baer and $\mathcal{K}$-(co)nonsingularity}\label{ksing}

In the module-theoretic context there are the following definitions which have appeared in the literature with different names (for example see \cite{rizvibaer} and \cite{tutuncu2010dual}). For the convenience of the reader, we present the definitions with the names from \cite[Definition 9.4]{crivei2016rickart} where the definitions are given in the context of abelian categories.

\begin{defn}
	Let $M$ be an  $R$-module.
	\begin{itemize}
		\item $M$ is $\mathcal{K}$-nonsingular if whenever $\Ker\varphi$ is essential in $M$ with $\varphi\in\End_{R}(M)$ it follows that $\varphi=0$.

		\item $M$ is $\mathcal{K}$-cononsingular if whenever $\varphi(N)\neq 0$ for all $\varphi\in\End_{R}(M)$ it follows that $N$ is essential in $M$.

		\item $M$ is $\mathcal{T}$-nonsingular if for every nonzero endomorphism $\varphi$ of $M$, $\varphi(M)$ is not small in $M$.

		\item $M$ is $\mathcal{T}$-cononsingular if for every non-small submodule $N$ of $M$, there exists a nonzero endomorphism $\varphi$ of $M$ such that $\varphi^{-1}(N)=M$.
	\end{itemize}
\end{defn}

Baer and dual-Baer modules have been characterized using these notions \cite[Theorem 2.12]{rizvibaer} and \cite[Theorem 2.14]{tutuncu2010dual}. Also, the counterpart in abelian categories was presented in \cite[Theorem 9.5]{crivei2016rickart}. As it can be seen, with the notion of linear morphism, it is possible to carry the above definitions to the lattice context.

\begin{defn}
	Let $\mathcal{L}$ be a complete modular lattice and let $\mathfrak{m}$ be a submonoid with zero of $\End_{lin}(\mathcal{L})$.
	\begin{enumerate}
		\item $\mathcal{L}$ is called \emph{$\mathfrak{m}$-$\mathcal{K}$-nonsingular} if whenever $\ker_\varphi$ is essential in $\mathcal{L}$ with $\varphi\in\mathfrak{m}$ it follows that $\varphi=0$.
		\item $\mathcal{L}$ is called \emph{$\mathfrak{m}$-$\mathcal{T}$-nonsingular} if whenever $\varphi(\mathbf{1})$ is superfluous in $\mathcal{L}$ with $\varphi\in\mathfrak{m}$ it follows that $\varphi=0$.
	\end{enumerate}
If the submonoid we are considering is $\End_{lin}(\mathcal{L})$, we will omit the $\mathfrak{m}$.
\end{defn}

\begin{defn}
	Let $\mathcal{L}$ be a complete modular lattice and let $\mathfrak{m}$ be a submonoid with zero of $\End_{lin}(\mathcal{L})$.
	\begin{enumerate}
		\item $\mathcal{L}$ is called \emph{$\mathfrak{m}$-$\mathcal{K}$-cononsingular} if whenever $\varphi(a)\neq \mathbf{0}$ for all $0\neq\varphi\in\mathfrak{m}$ it follows that $a$ is essential in $\mathcal{L}$.
		\item $\mathcal{L}$ is called \emph{$\mathfrak{m}$-$\mathcal{T}$-cononsingular} if $\varphi(\mathbf{1})\nleq a$ for all $0\neq\varphi\in\mathfrak{m}$, then $a$ is superfluous in $\mathcal{L}$.
	\end{enumerate}
If the submonoid we are considering is $\End_{lin}(\mathcal{L})$, we will omit the $\mathfrak{m}$.
\end{defn}

The next result follows from the above definitions.

\begin{prop}
	Let $M$ be an $R$-module. Then,
	\begin{enumerate}
		\item $M$ is $\mathcal{K}$-nonsingular if and only if $\Lambda(M)$ is $\mathfrak{E}_M$-$\mathcal{K}$-nonsingular.
		\item $M$ is $\mathcal{T}$-nonsingular if and only if $\Lambda(M)$ is $\mathfrak{E}_M$-$\mathcal{T}$-nonsingular.
		\item $M$ is $\mathcal{K}$-cononsingular if and only if $\Lambda(M)$ is $\mathfrak{E}_M$-$\mathcal{K}$-cononsingular.
		\item $M$ is $\mathcal{T}$-cononsingular if and only if $\Lambda(M)$ is $\mathfrak{E}_M$-$\mathcal{T}$-cononsingular.
	\end{enumerate}
\end{prop}

\begin{defn}
	Let $\mathcal{L}$ be a lattice.
	\begin{enumerate}
		\item $\mathcal{L}$ satisfies the condition \emph{($C_1$)} if for every $x\in\mathcal{L}$ there exists $c\in C(\mathcal{L})$ such that $x$ is essential in the interval $[\mathbf{0},c]$ \cite[Definition 1.1]{albu2016conditions}. 
		\item $\mathcal{L}$ satisfies the condition \emph{($D_1$)} if for every $x\in\mathcal{L}$ there exists $c\in C(\mathcal{L})$ with complement $c'$ such that $c\leq x$ and $x\wedge c'$ is superfluous in $\mathcal{L}$.
	\end{enumerate}
\end{defn}

\begin{lemma}
    Let $\mathcal{L}$ be  a complete modular lattice and let $\mathfrak{m}$ be a submonoid of $\End_{lin}(\mathcal{L})$ containing all the projections. If $\mathcal{L}$ satisfies ($C_1$) then $\mathcal{L}$ is $\mathfrak{m}$-$\mathcal{K}$-cononsingular.
\end{lemma}

\begin{proof}
    Let $a\in\mathcal{L}$ such that $\varphi(a)\neq\mathbf{0}$ for all $\varphi\in\mathfrak{m}$. By hypothesis, there exists $c\in C(\mathcal{L})$ such that $a$ is essential in $[\mathbf{0},c]$. Let $c'$ be a complement of $c$ in $\mathcal{L}$ and consider the projection $\pi_{c'}$. Then $\pi_{c'}(a)=\mathbf{0}$. This implies that $\pi_{c'}=0$ and so $c'=\mathbf{0}$. Therefore $c=\mathbf{1}$ and hence $a$ is essential in $\mathcal{L}$.
\end{proof}

\begin{cor}[{\cite[Lemma 2.13]{rizvibaer}}]
	Let $M$ be an $R$-module. If $M$ satisfies ($C_1$), then $M$ is $\mathcal{K}$-cononsingular.
\end{cor}

\begin{lemma}
	Let $\mathcal{L}$ be  a complete modular lattice and let $\mathfrak{m}$ be a submonoid of $\End_{lin}(\mathcal{L})$ containing all the projections. If $\mathcal{L}$ satisfies ($D_1$) then $\mathcal{L}$ is $\mathfrak{m}$-$\mathcal{T}$-cononsingular.
\end{lemma}

\begin{proof}
	Let $a\in\mathcal{L}$. By hypothesis, there exists $c\in C(\mathcal{L})$ with complement $c'$ such that $c\leq a$ and $a\wedge c'$ is superfluous in $\mathcal{L}$. Consider the projection $\pi_{c}$. Then $\pi_{c}(1)=\pi_c(c)=c\leq a$. This implies that $\pi_{c}=0$ and so $c=\mathbf{0}$. Therefore $c'=\mathbf{1}$ and hence $a$ is superfluous in $\mathcal{L}$.
\end{proof}

\begin{cor}[{\cite[Lemma 2.12]{tutuncu2010dual}}]
	Let $M$ be an $R$-module. If $M$ satisfies ($D_1$), then $M$ is $\mathcal{T}$-cononsingular.
\end{cor}

\begin{lemma}
	Let $\mathcal{L}$ be  a complete modular lattice and let $\mathfrak{m}$ be a submonoid of $\End_{lin}(\mathcal{L})$ containing all the projections. If $\mathcal{L}$ is $\mathfrak{m}$-$\mathcal{K}$-nonsingular and satisfies ($C_1$) then $\mathcal{L}$ is $\mathfrak{m}$-Baer.
\end{lemma}

\begin{proof}
	Let $\{\varphi_i\}_I\subseteq\mathfrak{m}$ and let $k=\bigwedge_I\ker_{\varphi_i}$. By hypothesis there exists an element $c\in\mathcal{L}$ with complement $c'$ such that $k$ is essential in $[\mathbf{0},c]$. This implies that $\ker_{\varphi_i}\wedge c$ is also essential in $[\mathbf{0},c]$ for all $i\in I$. Therefore $\ker_{\varphi_i\pi_c}=(\ker_{\varphi_i}\wedge c)\vee c'$ is essential in $\mathcal{L}$ for all $i\in I$. Since $\mathcal{L}$ is $\mathfrak{m}$-$\mathcal{K}$-nonsingular, $\varphi_i\pi_c=0$ for all $i\in I$. Hence $c\leq\ker_{\varphi_i}$ for all $i\in I$, that is, $c\leq k$. Thus, $k=c$. This proves that $\mathcal{L}$ is $\mathfrak{m}$-Baer.
\end{proof}

\begin{cor}[{\cite[Lemma 2.14]{rizvibaer}}]
	Let $M$ be an $R$-module. If $M$ is $\mathcal{K}$-nonsingular and satisfies ($C_1$), then $M$ is Baer.
\end{cor}

\begin{lemma}
	Let $\mathcal{L}$ be a complete modular lattice and let $\mathfrak{m}$ be a submonoid of $\End_{lin}(\mathcal{L})$ containing all the projections. If $\mathcal{L}$ is $\mathfrak{m}$-$\mathcal{T}$-nonsingular and satisfies ($D_1$) then $\mathcal{L}$ is dual-$\mathfrak{m}$-Baer.
\end{lemma}

\begin{proof}
	Let $\{\varphi_i\}_I\subseteq\mathfrak{m}$ and let $a=\bigvee_I\varphi_i(\mathbf{1})$. By hypothesis there exists an element $c\in\mathcal{L}$ with complement $c'$ such that $c\leq a$ and $a\wedge c'$ is superfluous in $\mathcal{L}$. This implies that $\pi_{c'}\varphi_i(\mathbf{1})=(\varphi_i(\mathbf{1})\vee c)\wedge c'$ is also superfluous in $\mathcal{L}$ for all $i\in I$. Since $\mathcal{L}$ is $\mathfrak{m}$-$\mathcal{T}$-nonsingular, $\pi_{c'}\varphi_i=0$ for all $i\in I$. Hence $\varphi_i(\mathbf{1})\leq c$ for all $i\in I$, that is, $a\leq c$. Thus, $a=c$. This proves that $\mathcal{L}$ is dual-Baer.
\end{proof}

\begin{cor}[{\cite[Lemma 2.11]{tutuncu2010dual}}]
	Let $M$ be an $R$-module. If $M$ is $\mathcal{T}$-nonsingular and satisfies ($D_1$), then $M$ is dual-Baer.
\end{cor}

\begin{lemma}
	Let $\mathcal{L}$ be  a complete modular lattice and let $\mathfrak{m}$ be a submonoid of $\End_{lin}(\mathcal{L})$. If $\mathcal{L}$ is $\mathfrak{m}$-Rickart then $\mathcal{L}$ is $\mathfrak{m}$-$\mathcal{K}$-nonsingular.
\end{lemma}
	
\begin{proof}
	Let $\varphi\in\mathfrak{m}$ such that $\ker_\varphi$ is essential in $\mathcal{L}$. Since $\mathcal{L}$ is $\mathfrak{m}$-Rickart, $\ker_\varphi$ has a complement. This implies that $\ker_\varphi=\mathbf{1}$ and hence $\varphi=0$.
\end{proof}

\begin{cor}[{\cite[Lemma 2.15]{rizvibaer}}]
	Let $M$ be an $R$-module. If $M$ is Baer, then $M$ is $\mathcal{K}$-nonsingular.
\end{cor}

\begin{lemma}
	Let $\mathcal{L}$ be  a complete modular lattice and let $\mathfrak{m}$ be a submonoid of $\End_{lin}(\mathcal{L})$. If $\mathcal{L}$ is dual-$\mathfrak{m}$-Rickart then $\mathcal{L}$ is $\mathfrak{m}$-$\mathcal{T}$-nonsingular.
\end{lemma}

\begin{proof}
	Let $\varphi\in\mathfrak{m}$ such that $\varphi(\mathbf{1})$ is superfluous in $\mathcal{L}$. Since $\mathcal{L}$ is dual-Rickart, $\varphi(\mathbf{1})$ has a complement. This implies that $\varphi(\mathbf{1})=\mathbf{0}$ and hence $\varphi=0$.
\end{proof}

\begin{cor}[{\cite[Corollary 2.3]{tutuncu2010dual}}]
	Let $M$ be an $R$-module. If $M$ is dual-Baer, then $M$ is $\mathcal{T}$-nonsingular.
\end{cor}

\begin{lemma}
	Let $\mathcal{L}$ be  a complete modular lattice and let $\mathfrak{m}$ be a submonoid of $\End_{lin}(\mathcal{L})$ containing all the projections. If $\mathcal{L}$ is $\mathfrak{m}$-Baer and $\mathfrak{m}$-$\mathcal{K}$-cononsingular then $\mathcal{L}$ satisfies ($C_1$).
\end{lemma}

\begin{proof}
	Let $a\in\mathcal{L}$. By Proposition \ref{baercar}, there exists an idempotent $\varepsilon\in\mathfrak{m}$ such that $\varepsilon(a)=\mathbf{0}$. Then $a\leq\ker_\varepsilon$ and $\ker_\varepsilon\in C(\mathcal{L})$. We claim that $a$ is essential in $[\mathbf{0},\ker_\varepsilon]$. Let $k'$ be a complement of $\ker_\varepsilon$. Let $\varphi\in\mathfrak{m}$ such that $\varphi(a\vee k')=\mathbf{0}$. Then $\varphi(k')=\mathbf{0}$ and $\varphi(a)=\mathbf{0}$. Since $\varphi(a)=\mathbf{0}$, $\varphi=\varphi\varepsilon$. Therefore, 
	\[\varphi(\mathbf{1})=\varphi(\ker_\varepsilon\vee k')=\varphi(\ker_\varepsilon)\vee\varphi(k')=\varphi\varepsilon(\ker_\varepsilon)\vee\varphi(k')=\mathbf{0}.\]
	Thus, $\varphi=0$. By hypothesis, $a\vee k'$ is essential in $\mathcal{L}$. Now, let $\mathbf{0}\leq b\leq \ker_\varepsilon$ such that $a\wedge b=\mathbf{0}$. Then $(a\vee k')\wedge b\leq (a\vee k')\wedge \ker_\varepsilon=a$. Hence $(a\vee k')\wedge b\leq a\wedge b=\mathbf{0}$. This implies that $b=\mathbf{0}$ proving the claim.
\end{proof}

\begin{cor}[{\cite[Lemma 2.16]{rizvibaer}}]
	Let $M$ be an $R$-module. If $M$ is Baer and $\mathcal{K}$-cononsingular, then $M$ satisfies ($C_1$).
\end{cor}

\begin{lemma}
	Let $\mathcal{L}$ be  a complete modular lattice and let $\mathfrak{m}$ be a submonoid of $\End_{lin}(\mathcal{L})$ containing all the projections. If $\mathcal{L}$ is dual-$\mathfrak{m}$-Baer and $\mathfrak{m}$-$\mathcal{T}$-cononsingular then $\mathcal{L}$ satisfies ($D_1$).
\end{lemma}

\begin{proof}
	Let $a\in\mathcal{L}$. By Proposition \ref{dbaercar}, there exists an idempotent $\varepsilon\in\mathfrak{m}$ such that $\varepsilon(\mathbf{1})\leq a$ and $\varepsilon(\mathbf{1})\in C(\mathcal{L})$. We have to show that $a\wedge\ker_{\varepsilon}$ is superfluous in $\mathcal{L}$. Let $\varphi\in\mathfrak{m}$ such that $\varphi(\mathbf{1})\leq a\wedge\ker_\varepsilon$. Then $\varphi=\varepsilon\varphi$. Therefore $\varphi(\mathbf{1})=\varepsilon\varphi(\mathbf{1})\leq \varepsilon(a\wedge\ker_\varepsilon)=\mathbf{0}$. Thus $\varphi=0$. Since $\mathcal{L}$ is $\mathcal{T}$-cononsingular, $a\wedge\ker_\varepsilon$ is superfluous in $\mathcal{L}$.
\end{proof}

\begin{cor}[{\cite[Proposition 2.13]{tutuncu2010dual}}]
	Let $M$ be an $R$-module. If $M$ is dual-Baer and $\mathcal{T}$-cononsingular, then $M$ satisfies ($D_1$).
\end{cor}

\begin{cor}\label{baercarK}
	The following conditions are equivalent for a complete modular lattice $\mathcal{L}$ and a submonoid $\mathfrak{m}$ of $\End_{lin}(\mathcal{L})$ containing all the projections:
	\begin{enumerate}[label=\emph{(\alph*)}]
		\item $\mathcal{L}$ is $\mathfrak{m}$-$\mathcal{K}$-nonsingular and satisfies ($C_1$).
		\item $\mathcal{L}$ is $\mathfrak{m}$-Baer and $\mathfrak{m}$-$\mathcal{K}$-cononsingular.
	\end{enumerate}
\end{cor}

\begin{cor}\label{dbaercarT}
	The following conditions are equivalent for a complete modular lattice $\mathcal{L}$ and a submonoid $\mathfrak{m}$ of $\End_{lin}(\mathcal{L})$ containing all the projections:
	\begin{enumerate}[label=\emph{(\alph*)}]
		\item $\mathcal{L}$ is $\mathfrak{m}$-$\mathcal{T}$-nonsingular and satisfies ($D_1$).
		\item $\mathcal{L}$ is dual-$\mathfrak{m}$-Baer and $\mathfrak{m}$-$\mathcal{T}$-cononsingular.
	\end{enumerate}
\end{cor}

\section{Decomposition and direct products of Rickart lattices}\label{decprod}

\begin{prop}
	Let $\mathcal{L}$ be a complete modular lattice and $\{a_1,...,a_n\}$ be a independent family of fully invariant elements of $\mathcal{L}$ such that $\mathbf{1}=\bigvee_{i=1}^na_i$. Then, $\mathcal{L}$ is Rickart if and only if $[\mathbf{0},a_i]$ are Rickart lattices for all $1\leq i\leq n$. 
\end{prop}

\begin{proof}
	$\Rightarrow$ It follows from Proposition \ref{compintric}.
	
	$\Leftarrow$ Suppose $[\mathbf{0},a_i]$ is a Rickart lattice for every $1\leq i\leq n$. Let $\varphi:\mathcal{L}\to\mathcal{L}$ be a linear morphism. For any $1\leq i\leq n$, $\varphi(a_i)\leq a_i$. This implies that $\varphi|:[\mathbf{0},a_i]\to[\mathbf{0},a_i]$ is a linear morphism. Therefore, $\ker_\varphi\wedge a_i$ is a complement in $[\mathbf{0},a_i]$. On the other hand,
	\[\varphi(\pi_{a_i}(\ker_\varphi))=\varphi\left(\left(\ker_\varphi\vee \bigvee_{i\neq j}a_j\right)\wedge a_i \right)\leq \varphi\left(\ker_\varphi\vee \bigvee_{i\neq j}a_j\right)=\varphi\left(\bigvee_{i\neq j}a_j\right).\]
	Hence,
	\[\varphi(\pi_{a_i}(\ker_\varphi))\leq \varphi\left(\bigvee_{i\neq j}a_j\right)\wedge\varphi(a_i)\leq\left(\bigvee_{i\neq j}a_j\right)\wedge a_i=\mathbf{0}.\]
	Thus, $\pi_{a_i}(\ker_\varphi)\leq\ker_\varphi\wedge a_i$ for all $1\leq i\leq n$. This implies that $\ker_\varphi=\bigvee_{i=1}^n\ker_\varphi\wedge a_i$. Then, $\ker_\varphi$ is a complement in $[\mathbf{0},\bigvee_{i=1}^na_i]=\mathcal{L}$.
\end{proof}

\begin{lemma}[{\cite[Proposition 6.8]{calugareanu2013lattice}}]\label{artif}
	Let $\mathcal{L}$ be a complete Artinian or Noetherian lattice. Then there exists an independent family $\{a_i\}_{i=1}^n$ of elements of $\mathcal{L}$ such that $\mathbf{1}=\bigvee_{i=1}^na_i$ and $[\mathbf{0},a_i]$ is indecomposable for all $1\leq i\leq n$.
\end{lemma}

\begin{defn}
    Let $\mathcal{L}$ be a complete lattice. The \emph{Socle} and the \emph{Radical} of $\mathcal{L}$ are the elements defined as follows:
    \[\Soc(\mathcal{L})=\bigvee\{x\in\mathcal{L}\mid x\text{ is an atom}\},\]
    \[\Rad(\mathcal{L})=\bigwedge\{x\in\mathcal{L}\mid x\text{ is a coatom}\}.\]
    \end{defn}

\begin{lemma}\label{ricind2}
	Let $\mathcal{L}$ be a complete modular lattice suppose $Soc(\mathcal{L})\neq\mathbf{0}$ and $\Rad(\mathcal{L})\neq\mathbf{1}$. Then $\mathcal{L}$ is indecomposable Rickart if and only if $\mathcal{L}\cong 2$ where $2$ is the complete lattice of two elements.
\end{lemma}

\begin{proof}
	By hypothesis, there exist $x,y\in \mathcal{L}$ such that $x$ is a coatom and $y$ is an atom. Then it can be defined the linear morphism $\varphi:\mathcal{L}\to \mathcal{L}$ as 
	\[\varphi(a)=\begin{cases}
	\mathbf{0} & \text{if } a\leq x \\
	y & \text{if } a=\mathbf{1}
	\end{cases}\]
	Hence, $\ker_\varphi=x$. By hypothesis $x$ has to be a complement but $\mathcal{L}$ is indecomposable. Thus $x=\mathbf{0}$ and $y=\mathbf{1}$, that is, $\mathcal{L}\cong 2$. The converse is obvious. 
\end{proof}

\begin{prop}\label{if2}
	Let $\mathcal{L}$ be a complete modular lattice. Then $\mathcal{L}$ is Rickart of finite length if and only if there exists an independent family $\{a_i\}_{i=1}^n$ of elements of $\mathcal{L}$ such that $\mathbf{1}=\bigvee_{i=1}^na_i$ and $[\mathbf{0},a_i]\cong 2$ for all $1\leq i\leq n$.
\end{prop}

\begin{proof}
	$\Rightarrow$ By Lemma \ref{artif}, there exists an independent family $\{a_i\}_{i=1}^n$ of elements of $\mathcal{L}$ such that $\mathbf{1}=\bigvee_{i=1}^na_i$ and each interval $[\mathbf{0},a_i]$ is indecomposable. By Proposition \ref{compintric} and Lemma \ref{ricind2}, $[\mathbf{0},a_i]\cong 2$ for all $1\leq i\leq n$.
	
	$\Leftarrow$ It is clear that $\mathcal{L}$ has finite length. We claim that $\mathcal{L}$ is complemented. We prove the claim by induction in the length of the independent family $\{a_i\}_{i=1}^n$. If $n=1$, then $\mathcal{L}=[\mathbf{0},a_1]\cong 2$ which is complemented. Now, suppose that the interval $[\mathbf{0},a_1\vee\cdots\vee a_{n-1}]$ is complemented and $\bigvee_{i=1}^na_i=\mathbf{1}$. Let $x\in\mathcal{L}$. Set $y=a_1\vee\cdots\vee a_{n-1}$ and consider $x\vee y$. We have that $y\leq x\vee y\leq \mathbf{1}$ and by hypothesis $[y,\mathbf{1}]\cong[\mathbf{0},a_n]\cong 2$. This implies that $x\leq y$ or $x\vee y=\mathbf{1}$. If $x\leq y$ then $x$ has a complement by induction hypothesis. Suppose $x\vee y=\mathbf{1}$. Since $x\wedge y\leq y$ then there exists $\mathbf{0}\leq b\leq y$ such that $b\wedge(x\wedge y)=\mathbf{0}$ and $b\vee(x\wedge y)=y$. Therefore $b\wedge x=b\wedge y\wedge x=\mathbf{0}$ and $\mathbf{1}=x\vee y=x\vee(b\vee(x\wedge y))=x\vee((b\vee x)\wedge b)=x\vee b$. Thus, $x$ is a complement proving the claim.
\end{proof}

The last proofs of Lemma \ref{ricind2} and Proposition \ref{if2} can be applied to modules in particular cases as follows.

\begin{cor}\label{ricmodsimp}
	Suppose that
	\begin{enumerate}
		\item $R$ is a left local ring and $M$ is an $R$-module with $\Soc(M)\neq 0$ and $\Rad(M)\neq M$, or
		\item $M$ is a Kasch $R$-module and $\Rad(M)\neq M$, or
		\item $M$ contains an $M$-generated simple module.
	\end{enumerate}
Then, $M$ is indecomposable Rickart if and only if $M$ is simple.
\end{cor}

\begin{cor}\label{ricmodss}
	Suppose that
	\begin{enumerate}
		\item $R$ is a local ring and $M$ is an $R$-module of finite length, or
		\item $M$ is a Kasch $R$-module of finite length.
	\end{enumerate}
	Then, $M$ is Rickart if and only if $M$ is semisimple.
\end{cor}

\begin{proof}
	Since $M$ has finite length, $M=M_1\oplus\cdots\oplus M_\ell$ for some $\ell>0$ with each $M_i$ an indecomposable Rickart module of finite length. If $R$ is local, the result is clear. Suppose $M$ is a Kasch module and let $N$ be a maximal submodule of $M_i$. Then, there exists an embedding $M_i/N\to M$. It follows from \cite[Theorem 2.6]{leedirect} (see also Proposition \ref{ricdirsumsub} below) that $M_i$ is $M$-Rickart, hence $N$ is direct summand of $M_i$. Thus, $M_i$ is simple.
\end{proof}

\begin{cor}\label{sumric}
	Let $\mathcal{L}$ be a finite modular lattice and $\{a_1,...,a_n\}$ be a independent family of elements of $\mathcal{L}$ such that $\mathbf{1}=\bigvee_{i=1}^na_i$. The following conditions are equivalent:
	\begin{enumerate}[label=\emph{(\alph*)}]
		\item $\mathcal{L}$ is Rickart.
		\item $[\mathbf{0},a_i]$ are Rickart lattices for all $1\leq i\leq n$.
	\end{enumerate}
\end{cor}

\begin{proof}
	(a)$\Rightarrow$(b) is clear. For the converse, since each interval $[\mathbf{0},a_i]$ is a Rickart lattice of finite length, there exists an independent family $\{b_{ij}\}_{j=1}^{\ell_i}$ in the interval $[\mathbf{0},a_i]$ such that $\bigvee_{j=1}^{\ell_i}b_{ij}=a_i$ and $[\mathbf{0},b_{ij}]\cong 2$ for all $1\leq j\leq\ell_i$ by Proposition \ref{if2}. Then, the family $\{b_{ij}\mid 1\leq i\leq n,\;1\leq j\leq \ell_i\}$ is an independent family of elements of $\mathcal{L}$ such that $\bigvee\{b_{ij}\mid 1\leq i\leq n,\;1\leq j\leq \ell_i\}=\mathbf{1}$ and $[\mathbf{0},b_{ij}]\cong 2$. It follows from Proposition \ref{if2} that $\mathcal{L}$ is Rickart.
\end{proof}

In general there might be indecomposable Artinian non Noetherian Rickart lattices as the following example shows.

\begin{example}
	Consider the lattice $\mathcal{L}$ given by the following diagram.
	
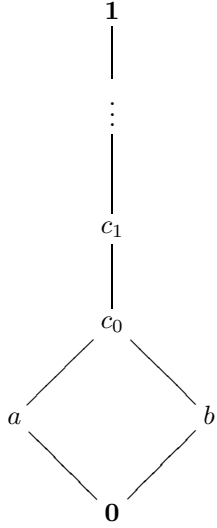
\[\xymatrix{ & \mathbf{1}\ar@{-}[d] &  \\
	 & \vdots\ar@{-}[d] & \\
 	 & c_1\ar@{-}[d] & \\
  	 & c_0\ar@{-}[dl]\ar@{-}[dr] & \\
   	 a\ar@{-}[dr] & & b\ar@{-}[dl] \\
     & \mathbf{0} & }\]
	Then $\mathcal{L}$ is Artinian and indecomposable. Note that any nonzero linear morphism $\varphi:\mathcal{L}\to\mathcal{L}$ must be an isomorphism. Thus, $\mathcal{L}$ is Rickart. 
\end{example}

Given a family of complete modular lattices $\{\mathcal{L}_i\}_I$, the Cartesian product $\prod_I\mathcal{L}_i$ is a complete modular lattice with supremum and infimimum pointwise. The cartesian product is defined as the set of functions 
\[\prod_I\mathcal{L}_i=\left\lbrace f:I\to\bigcup_I\mathcal{L}_i\mid f(i)\in\mathcal{L}_i\right\rbrace \]
Then, $f\leq g$ if and only if $f(i)\leq g(i)$ for all $i\in I$, $(f\wedge g)(i)=f(i)\wedge g(i)$ and $(f\vee g)(i)=f(i)\vee g(i)$. For every $i\in I$, we have the canonical projections $\pi_i:\prod_I\mathcal{L}_i\to\mathcal{L}_i$ defined as $\pi_i(f)=f(i)$.

\begin{prop}
	Let $\{\mathcal{L}_i\}_I$ be a family of complete modular lattices. Then the projections $\pi_i:\prod_I\mathcal{L}_i\to \mathcal{L}_i$ are linear morphisms.
\end{prop}

\begin{proof}
	Fix $j\in I$, and consider $\pi_j$. Let $k_j\in\prod_I\mathcal{L}_i$ be the function defined as follows:
	\[k_j(i)=\begin{cases}
	\mathbf{0}\;\text{ if }\;i=j \\
	\mathbf{1}\;\text{ if }\;i\neq j
	\end{cases}\]
	Then $\pi_j(k_j)=k_j(j)=\mathbf{0}\in\mathcal{L}_j$. Also, 
	\[\pi_j(f\vee k_j)=(f\vee k_j)(j)=f(j)\vee k_j(j)=f(j)\vee\mathbf{0}=f(j)=\pi_j(f).\]
	For $a\in \mathcal{L}_j$, define
	\[f_a(i)=\begin{cases}
	a\;\text{ if }\; i=j \\
	\mathbf{1}\;\text{ if }\; i\neq j
	\end{cases}\]
	Let $\rho:\mathcal{L}_j\to [k_j,\mathbf{1}]\subseteq\prod_I\mathcal{L}_i$ given by $\rho(a)=f_a$. It is clear that $\rho$ is a lattice morphism and $k_j\leq f_a$ for all $a\in\mathcal{L}_j$. Consider $k_j\leq g\leq \mathbf{1}$. Then,
	\[\rho\pi_j(g)=\rho(g(j))=f_{g(j)}\]
	and 
	\[f_{g(j)}(i)=\begin{cases}
	g(j)\;\text{ if }\; i=j \\
	\mathbf{1}\;\text{ if }\; i\neq j
	\end{cases}=g(i)\]
	Thus, $\rho\pi_j(g)=g$. Now consider $a\in\mathcal{L}_j$. Then,
	\[\pi_j\rho(a)=\pi_j(f_a)=f_a(j)=a.\]
	Therefore $\pi_j$ induces an isomorphism of lattices $\overline{\pi_j}:[k_j,\mathbf{1}]\to\mathcal{L}_j$. Hence $\pi_j$ is a linear morphism with kernel $\ker_{\pi_j}=k_j$.
\end{proof}

\begin{defn}
	Let $\mathcal{L}$ and $\mathcal{M}$ be lattices. It is said that $\mathcal{L}$ is $\mathcal{M}$-Rickart if $\ker_\varphi$ is complemented in $\mathcal{L}$ for every linear morphism $\varphi:\mathcal{L}\to \mathcal{M}$.
\end{defn}

\begin{prop}\label{ricdirsumsub}
	Let $\mathcal{L}$ and $\mathcal{M}$ be complete modular lattices and suppose $\mathcal{L}$ is $\mathcal{M}$-Rickart. If $a\in\mathcal{L}$ is a complement and $x\in\mathcal{M}$, then $[\mathbf{0},a]$ is $[\mathbf{0},x]$-Rickart.
\end{prop}

\begin{proof}
	It follows from Proposition \ref{exmorf} and the fact that the inclusion $\iota:[\mathbf{0},x]\to\mathcal{M}$ is a linear morphism.
\end{proof}

\begin{prop}
	Let $\mathcal{L}$ and $\mathcal{M}$ be complete modular lattices and $\{a_i\}_I$ be a independent family of elements of $\mathcal{M}$ with $\mathbf{1}=\bigvee_{i\in I}a_i$.
	\begin{enumerate}
		\item If $\mathcal{L}$ has CIP, $\mathcal{L}$ is $\mathcal{M}$-Rickart if and only if $\mathcal{L}$ is $[\mathbf{0},a_i]$-Rickart for all $i\in I$ with $I=\{1,...,n\}$.
		\item If $\mathcal{L}$ is upper continuous and has SCIP, $\mathcal{L}$ is $\mathcal{M}$-Rickart if and only if $\mathcal{L}$ is $[\mathbf{0},a_i]$-Rickart for all $i\in I$ for any index set $I$.
	\end{enumerate}
\end{prop}

\begin{proof}
	(1)$\Rightarrow$ It follows from Proposition \ref{compintric}.
	
	$\Leftarrow$ Let $\varphi:\mathcal{L}\to \mathcal{M}$ be a linear morphism. Then, we have the linear morphisms $\pi_{a_i}\varphi:\mathcal{L}\to [\mathbf{0},a_i]$. By hypothesis, $\ker_{\pi_{a_i}\varphi}$ is complemented in $\mathcal{L}$ for every $i\in I$. We claim that $\ker_\varphi=\bigwedge_I\ker_{\pi_{a_i}\varphi}$. It is clear that $\ker_\varphi\leq\bigwedge_I\ker_{\pi_{a_i}\varphi}$. Suppose $I$ is finite. It follows from Proposition \ref{fipi}.(1) that
	\[\varphi\left( \bigwedge_I\ker_{\pi_{a_i}\varphi}\right)\leq \bigvee_{j\in I}\pi_j\left(\varphi\left( \bigwedge_I\ker_{\pi_{a_i}\varphi}\right) \right)\leq\bigvee_{j\in I}\pi_{a_j}\varphi(\ker_{\pi_{a_j}\varphi})=0.\]
	This proves the claim. Now, since $\mathcal{L}$ has the CIP, $\ker_\varphi$ is complemented in $\mathcal{L}$.
	
	(2) The proof is similar using Proposition \ref{fipi}.(2).
\end{proof}

\begin{prop}
	Let $\{\mathcal{L}_i\}_I$ be a family of complete modular lattices. If $\prod_I\mathcal{L}_i$ is a Rickart lattice, then $\mathcal{L}_i$ is $\mathcal{L}_j$-Rickart for all $i,j\in I$. If $I$ is finite, then the converse is true.
\end{prop}

\begin{proof}
	Fix $j,k\in I$. Let $\varphi:\mathcal{L}_k\to\mathcal{L}_j$ be a linear morphism. Define $\psi:\prod_I\mathcal{L}_i\to\prod_I\mathcal{L}_i$ as 
	\[\psi(f)(i)=\begin{cases}
	\varphi(f(k))\;\text{ if }\; i=j \\
	\mathbf{0}\;\text{ if }\; i\neq j.
	\end{cases}\]
	It is not difficult to see that $\psi$ is a linear morphism with 
	\[\ker_\psi(i)=\begin{cases}
	\ker_\varphi\;\text{ if }\; i=k \\
	\mathbf{1}\;\text{ if }\; i\neq  k.
	\end{cases}\]
	By hypothesis, $\ker_\psi$ has a complement $f$ in $\prod_I\mathcal{L}_i$. Then $\mathbf{0}=(\ker_{\psi}\wedge f)(k)=\ker_{\psi}(k)\wedge f(k)=\ker_{\varphi}\wedge f(k)$ and $\mathbf{1}=(\ker_{\psi}\vee f)(k)=\ker_{\psi}(k)\vee f(k)=\ker_{\varphi}\vee f(k)$ in $\mathcal{L}_k$. Thus, $\ker_{\varphi}$ is a complement.
	
	Now assume $I=\{1,...,n\}$. Suppose that $\mathcal{L}_i$ is an $\mathcal{L}_j$-Rickart lattice for all $1\leq i,j\leq n$. Let $\varphi:\prod_{i=1}^n\mathcal{L}_i\to\prod_{i=1}^n\mathcal{L}_i$ be a linear morphism. For each $1\leq j\leq n$, consider the inclusions $\iota_j:\mathcal{L}_j\to\prod_{i=1}^n\mathcal{L}_i$ given by
	\[\iota_j(a)(i)=\begin{cases}
	a\;\text{ if }\; i=j \\
	\mathbf{0}\;\text{ if }\; i\neq j.
	\end{cases}\]
	Then we have a linear morphism $\pi_i\varphi\iota_j:\mathcal{L}_j\to\mathcal{L}_i$ for all $1\leq i,j\leq n$. Note that if $(a_1,...,a_n)\in\prod_{i=1}^n\mathcal{L}_i$ then $(a_1,...,a_n)=\bigvee_{i=1}^n\iota_i(a_i)$. Since $\varphi$ is a linear morphism $\varphi(a_1,...,a_n)=\varphi(\bigvee_{i=1}^n\iota_i(a_i))=\bigvee_{i=1}^n\varphi\iota_i(a_i)$ by \cite[Proposition 1.3]{albu2013category}. Hence
	\[\varphi(a_1,...,a_n)=\left(\bigvee_{i=1}^n\pi_1\varphi\iota_i(a_i),...,\bigvee_{i=1}^n\pi_n\varphi\iota_i(a_i) \right).\]
	Thus, $\varphi(a_1,...,a_n)=\mathbf{0}$ if and only if $\pi_j\varphi\iota_i(a_i)=\mathbf{0}$ for all $1\leq i,j\leq n$. Let $k_{ij}=\ker_{\pi_j\varphi\iota_i}$ and let $k_i=\bigwedge_{j=1}^n=k_{ij}$ for $1\leq i,j\leq n$. It is clear that $\varphi(k_1,...,k_n)=\mathbf{0}$. Now, if $\varphi(a_1,...,a_n)=\mathbf{0}$, then $a_i\leq\bigwedge_{j=1}^n=k_{ij}=k_i$. Thus, $\ker_\varphi=(k_1,...,k_n)$. By Proposition \ref{riccipssp}, $\mathcal{L}_i$ has the CIP for all $1\leq i\leq n$. Therefore each $k_i$ has a complement $x_i$ in $\mathcal{L}_i$. Thus $(x_1,...,x_n)$ is a complement for $\ker_\varphi$ in $\prod_{i=1}^n\mathcal{L}_i$.
\end{proof}

\bibliographystyle{acm}
\bibliography{biblio}

\end{document}